%% file: adaptivenet_new.tex
\documentclass[onefignum,onetabnum]{siamonline220329}
\usepackage{amsmath,amssymb,ifthen}
\usepackage{verbatim}
\usepackage{psfrag}
\usepackage[margin=2cm]{geometry}
\usepackage{tikz,pgf}
\usepackage{graphicx,psfrag,subfigure}
\usetikzlibrary{arrows,automata}
\usetikzlibrary{positioning}
\usepackage{tikz-3dplot}
\input{mymacros.tex}

\def\rev#1{#1}

\headers{Towards optimal hierarchical training of neural networks}{M. Feischl, A. Rieder, F. Zehetgruber}

\title{Towards optimal hierarchical training\\ of neural networks\thanks{Submitted to the editors DATE.
\funding{Funded by the Deutsche Forschungsgemeinschaft (DFG, German Research Foundation) -- Project-ID 258734477 -- SFB 1173, the Austrian Science Fund (FWF)
under the special research program Taming complexity in PDE systems (grant SFB F65) as well as projects I6667-N and P36150. Funding was received also from the European Research Council (ERC) under the
European Union’s Horizon 2020
research and innovation programme (Grant agreement No. 101125225).}}}

\author{Michael Feischl\thanks{Institute of Analysis and Scientific Computing, TU Wien 
  (\email{michael.feischl@tuwien.ac.at}).}
\and Alexander Rieder\thanks{Institute of Analysis and Scientific Computing, TU Wien 
  (\email{alexander.rieder@tuwien.ac.at}).}
\and Fabian Zehetgruber\thanks{Institute of Analysis and Scientific Computing, TU Wien 
  (\email{fabian.zehetgruber@tuwien.ac.at}).}}

\begin{document}
\maketitle
\begin{abstract}
    We propose a hierarchical training algorithm for standard feed-forward neural networks that adaptively extends the network architecture as soon as the optimization reaches a stationary point. By solving small (low-dimensional) optimization problems, the extended network provably escapes any local minimum or stationary point. Under some assumptions on the approximability of the data with \emph{stable} neural networks, we show that the algorithm achieves an optimal convergence rate $s$ in the sense that ${\rm loss}\lesssim \#{\rm parameters}^{-s}$. As a byproduct, we obtain computable indicators which judge the optimality of the training state of a given network and derive a new notion of generalization error.
\end{abstract}

\section{Introduction}
\noindent
One of the most pressing problems in deep learning is the question of how to design good networks for a given task and how to train them efficiently. 
The design of network architectures is often driven by intuition and/or experiment. This is particularly true for machine learning tasks in language modelling or computer vision and hence automated methods are in demand~\cite{nas_google}. 
Particularly in mathematical applications of deep learning such as function approximation~\cite{jentzen_approx,schwab_approx,schwab_approx2,schwab_approx3}, solving partial differential equations~\cite{deep_ritz0,deep_ritz1,weakan}, or inverse problems~\cite{schwab_inverse,berg_inverse}, we have the unique advantage of constructive approximation results that show that a certain architecture will, with the right weights, achieve a certain approximation quality. However, even if we
know the optimal architecture, it is still an open question of how to find the optimal weights in reasonable
compute time. The conventional approach of employing variants of (stochastic) gradient descend is known
to sometimes converge to weights that result in networks with much worse approximation qualities than
what is predicted from approximation theory. Examples of this gap can be found everywhere in the
literature, even with mathematical justification~\cite{t2p-gap}. This is in stark contrast to classical approximation
methods, such as interpolation, spectral methods, or finite element/volume methods. There, for many
problem sets, algorithms are known that design optimal approximation architectures (e.g., the mesh for
a finite element method) and also equip them with the optimal parameters (e.g., the coefficients of the
polynomial on each mesh element). The design of the architecture and the determination of the parameters
go hand-in-hand. One prominent example of this is the adaptive finite element method. First, a rough approximation $u_h$ to the exact solution $u$ of a given problem (usually a partial differential equation) is computed on a coarse mesh. Using this first approximation, the algorithm decides where to refine the mesh and updates the solution. This is repeated until the approximation is accurate enough. 
The optimality results known for such algorithms are of the following form: Assume that there exists a mesh (that, in general, cannot be computed practically) on which a good approximation $u_h^{\rm opt}\approx u$ exists, then the algorithm will find a similarly good mesh with a comparable amount of degrees of freedom as the optimal mesh on which it computes an approximation $u_h$ of comparable quality, i.e.,
\begin{align*}
    \norm{u-u_h}{}\leq C\norm{u-u_h^{\rm opt}}{}\quad \text{and}\quad
    \text{compute-time}(u_h)\leq C\,\text{compute-time}(u_h^{\rm opt}),
\end{align*}
{for some constant $C>0$ that depends only on secondary characteristics of the method but is independent of the solution $u$, its approximation and the number of degrees of freedom}
For examples of such results, we refer to the seminal works~\cite{bdd,stevenson07,ckns} in which the Poisson problem was tackled. The ideas in these works were generalized to many other model problems, e.g.,~\cite{ks,cn} for conforming methods,~\cite{BeMao10,ncstokes3}    
for non-conforming methods, and~\cite{LCMHJX} for mixed formulations. Even, non-local problems such as integral equations~\cite{gantumur,fkmp} and fractional problems~\cite{frac}. For an overview of the proof techniques, we refer to~\cite{axioms}.

\medskip

The big advantage of this line of reasoning is that one does not need to know the approximability of the exact solution $u$ and can still guarantee that the approximation is (quasi-) optimal for a given computational budget.

\bigskip

The present work pursues the question whether similar results are possible for the training of deep neural networks and gives some promising first answers.
To that end, we propose an algorithm that hierarchically extends the network architecture by adding new neurons in each layer as illustrated in an exemplary fashion below. This is done in such a way that the realization of the network is unchanged at first, i.e., the new weights (in red and blue) are set to zero. Then, a specific initialization of the new weights is computed by solving a small optimization problem. Under the assumption that a neural network of a certain size exists that strictly improves the 
loss function, we can show that this initialization reduces the loss by a factor that is proportional to the loss reduction if the best possible (but in general unknown) initialization would be chosen.

\vspace{2mm}
\begin{minipage}{0.31\linewidth}
\centering
\begin{tikzpicture}[
  node distance=1cm,
  neuron/.style={draw, circle, minimum size=0.2cm},
  input/.style={neuron,fill=green!30},
  hidden/.style={neuron,fill=green!30},
  output/.style={neuron,fill=green!30}
]
\node[input] (i1) {};
\node[input, below=of i1] (i2) {};
\node[input, below=of i2] (i3) {};

\node[hidden, right=of i1] (h1) {};
\node[hidden, below=of h1] (h2) {};
\node[hidden, below=of h2] (h3) {};

\node[hidden, right=of h1] (h11) {};
\node[hidden, below=of h11] (h12) {};
\node[hidden, below=of h12] (h13) {};

\node[output, right=of h12] (o1) {};

\foreach \i in {1,2,3} {
  \foreach \j in {1,2,3} {
    \draw[->] (i\i) -- (h\j);
  }
}

\foreach \i in {1,2,3} {
  \foreach \j in {1,2,3} {
    \draw[->] (h\i) -- (h1\j);
  }
}

\foreach \j in {1,2,3} {
  \draw[->] (h1\j) -- (o1);
}
\end{tikzpicture}
\end{minipage}
\begin{minipage}{0.31\linewidth}
\centering
\begin{tikzpicture}[
  node distance=0.7cm and 1cm,
  neuron/.style={draw, circle, minimum size=0.2cm},
  input/.style={neuron},
  hidden/.style={neuron},
  output/.style={neuron}
]
\node[input, fill=green!30] (i1) {};
\node[input, below=of i1, fill=green!30] (i2) {};
\node[input, below=of i2, fill=green!30] (i3) {};

\node[hidden, right=of i1, fill=green!30] (h2) {};
\node[hidden, above=of h2, fill=purple!30] (h1) {};
\node[hidden, below=of h2, fill=green!30] (h3) {};
\node[hidden, below=of h3, fill=green!30] (h4) {};
\node[hidden, below=of h4, fill=purple!30] (h5) {};

\node[hidden, right=of h1, fill=purple!30] (h11) {};
\node[hidden, below=of h11, fill=green!30] (h12) {};
\node[hidden, below=of h12, fill=green!30] (h13) {};
\node[hidden, below=of h13, fill=green!30] (h14) {};
\node[hidden, below=of h14, fill=purple!30] (h15) {};

\node[output, right=of h13, fill=green!30] (o1) {};

\foreach \i in {1,2,3} {
  \foreach \j in {1,2,3,4,5} {
    \draw[->] (i\i) -- (h\j);
  }
}

\foreach \i in {1,2,3,4,5} {
  \foreach \j in {1,2,3,4,5} {
    \draw[->] (h\i) -- (h1\j);
  }
}

\foreach \j in {1,2,3,4,5} {
  \draw[->] (h1\j) -- (o1);
}
\end{tikzpicture}
\end{minipage}
\begin{minipage}{0.31\linewidth}
\centering
\begin{tikzpicture}[
  node distance=0.5cm and 1cm,
  neuron/.style={draw, circle, minimum size=0.2cm},
  input/.style={neuron},
  hidden/.style={neuron},
  output/.style={neuron}
]
\node[input, fill=green!30] (i1) {};
\node[input, below=of i1, fill=green!30] (i2) {};
\node[input, below=of i2, fill=green!30] (i3) {};

\node[hidden, right=of i1, fill=green!30] (h3) {};
\node[hidden, above=of h3, fill=purple!30] (h2) {};
\node[hidden, above=of h2, fill=blue!30] (h1) {};
\node[hidden, below=of h3, fill=green!30] (h4) {};
\node[hidden, below=of h4, fill=green!30] (h5) {};
\node[hidden, below=of h5, fill=purple!30] (h6) {};
\node[hidden, below=of h6, fill=blue!30] (h7) {};

\node[hidden, right=of h2, fill=purple!30] (h11) {};
\node[hidden, below=of h11, fill=green!30] (h12) {};
\node[hidden, below=of h12, fill=green!30] (h13) {};
\node[hidden, below=of h13, fill=green!30] (h14) {};
\node[hidden, below=of h14, fill=purple!30] (h15) {};

\node[output, right=of h13, fill=green!30] (o1) {};

\foreach \i in {1,2,3} {
  \foreach \j in {1,2,3,4,5,6,7} {
    \draw[->] (i\i) -- (h\j);
  }
}

\foreach \i in {1,2,3,4,5,6,7} {
  \foreach \j in {1,2,3,4,5} {
    \draw[->] (h\i) -- (h1\j);
  }
}

\foreach \j in {1,2,3,4,5} {
  \draw[->] (h1\j) -- (o1);
}
\end{tikzpicture}
\end{minipage}

\vspace{2mm}
\noindent
To rigorously establish the loss reduction, we introduce the concept of \emph{stable networks}. The notion quantifies the amount of cancellation that is present in a network with given weights. Roughly speaking, we show that the existence of stable networks that achieve small loss guarantees that the proposed hierarchical training algorithm will find networks of similar size with similar loss. We argue that the notion of stability is necessary in the sense that if a certain loss functional can only be made small by very unstable networks (i.e., a small change in the weights results in a large change of the output), there is little hope for any training algorithm to find the correct weights for such a network. Practical experience also seems to suggest that this stability is inherent in many deep learning applications. For example, it is common practice to do inference or training of large networks in single precision~\cite{mixedprec} in order the increase performance (or even less bits per weight~\cite{onebit}).

As a corollary, we obtain a computable quantity that can be used to judge the training state of a given network. The quantity measures how far a network of a given size is from the smallest possible loss for this network architecture. This is similar to the well-known C\'ea lemma for Galerkin methods, which states that a computed approximation is, up to a factor, as good as the best possible approximation within the given setting.
The difference to the present setting is that we do not know the factor a~priori, but compute a quantity that is related to this factor. In that sense, we have an a~posteriori type C\'ea lemma for training neural networks.

Finally, we show that our hierarchical training algorithm results in improved estimates on the generalization error compared to direct algorithms. Moreover, we introduce the concept of \emph{optimal generalization} which turns out to be necessary in order to achieve small generalization error and allows us to make statements about the generalization error without quantifying the distance of training data and possible inputs of the network. This is particularly interesting as current machine learning applications such as Large Language Models {(e.g.,\cite{gpt4})} demonstrate remarkable generalization over inputs that are far away from the training data.

\subsection{Contributions of this work}

\begin{itemize}
    \item We propose a hierarchical training algorithm (Algorithm~\ref{alg:adaptive}) that increases the network size during training and (under certain assumptions) achieves an asymptotically quasi-optimal loss for a given number of trainable parameters (Theorem~\ref{thm:main}). Roughly speaking, this means that if for some $\beta>0$ and all $n\in\N$ there exist networks that achieve a loss smaller than $n^{-\beta}$ with a number of trainable parameters of order $n$, then Algorithm~\ref{alg:adaptive} will find a network with $n$ parameters that achieves a loss of $Cn^{-\beta}$ (or with some reduced rate $\widetilde \beta$ in some cases).
    
    \item We propose a computable indicator that quantifies how well a network with a given number of parameters is trained and how much improvement of the loss can be achieved by increasing the size of the network (Theorem~\ref{thm:cea}).
    
    \item We show that the hierarchical training algorithm (Algorithm~\ref{alg:adaptive}) has improved generalization error if the training data is well distributed. Moreover, we propose the notion of \emph{optimal generalization} that does not require the sometimes unrealistic assumption that the training data is well distributed among the set of possible inputs. We show that, under some assumptions, Algorithm~\ref{alg:adaptive} has optimal generalization in this sense.
    
    \item We show that \emph{stable} neural networks follow predictable scaling laws in the sense that existence of a large network that approximates the given data well always implies the existence of smaller networks whose approximation quality scales relative to their size difference.

\end{itemize}

\subsection{Other approaches and related work}

The field of \emph{neural architecture search} tries to automatically design neural network architectures that fit a given problem. After the seminal work~\cite{nas_google} the uses reinforcement learning to find good architectures, a wide range of methods was developed employed and we only give a few examples for algorithms that use genetic programming~\cite{nas1,nas2},
evolutionary algorithms~\cite{nas3} and Bayesian optimization~\cite{baysiannas} (see also~\cite{nasoverview} for an overview on the topic).
The architectures considered are much more involved then in the present paper, but a rigorous convergence analysis is often foregone in favor of extensive numerical confirmation. In contrast, the present work is restricted to simple feed forward networks, but proves a convergence guarantee with certain rates. In somewhat similar fashion, the work~\cite{greedy} proves convergence guarantees for greedy algorithms in dictionary learning. However, the authors have to assume a certain decay of coefficients in the target which is hard to guarantee in practical applications (but their results are stronger). The work~\cite{xu2025randomizedgreedyalgorithmsneural} proposes a similar greedy algorithm an proves dimension independent convergence rates for smooth approximation targets. 

{The work~\cite{adanet} also adaptively designs neural network architectures but with the goal to reduce the generalization error (see also~\cite{deepboost} for a similar approach in ensemble learning methods). Their architecture search algorithm, however, is based on a search in a high dimensional parameter space which must be done with heuristics.}

In terms of designing neural network architectures for mathematical problems such as approximation or solving partial differential equations, there is a large body of work and we mention only a few examples for approximation~\cite{jentzen_approx,schwab_approx,schwab_approx2,schwab_approx3}, for solving partial differential equations~\cite{deep_ritz0,deep_ritz1,weakan}, and for inverse problems~\cite{schwab_inverse,berg_inverse}. Those works either explicitly construct neural network architectures with certain (approximation) qualities or give methods on how to rewrite mathematical problems as optimization problems that can be solved by training neural networks. What is missing is a rigorous treatment of how to find the good weights for the given architectures in practical time constraints. The work~\cite{t2p-gap} even show the impossibility of finding good weights in certain situations.

When it comes to quantification of the approximation error of a trained neural network, there are approaches inspired by classical finite element error estimation in~\cite{canuto1,canuto2} and $\Gamma$-convergence in~\cite{babis}.

Finally, there is extensive literature that gives convergence guarantees for (stochastic) gradient descent to global minima (even with linear rate) under different assumptions, for example if the network is highly overparametrized~\cite{op1,op2} or for infinite width networks~\cite{ntk}.

\rev{The proposed hierarchical algorithm can be interpreted as block-coordinate optimization with an adaptively enlarged parameter space: each added subnetwork forms a new parameter block that is optimized while the existing parameters are held fixed, interleaved with joint retraining \cite{tseng2001,wright2015}. Its residual-based initialization provides a criterion for selecting new blocks. Incremental gradient methods exploit a complementary decomposition, updating parameters using individual summands of the empirical loss \cite{bertsekas2011}, and could therefore serve as inner optimization routines. These connections suggest combining adaptive architecture growth with inexpensive incremental updates, provided the required descent conditions are maintained. Under stability assumptions, our analysis relates the improvement obtained from newly introduced blocks to network complexity, providing a basis for investigating such combinations.}


\section{Notation and Definitions}
We consider standard {fully connected} feed-forward networks with one fixed activation. While more intricate network structures would certainly be possible, we aim for a straightforward presentation of the main ideas.

Since we aim to expand the network during the training, it is important to keep track of the network architecture. To that end, we introduce the architecture $\chi = (w_0,\ldots,w_{d+1})$ which stores the depth $d$ (i.e., the number of layers) as well as the widths $w_i$ (i.e., the number of neurons) of the individual layers.
For simplicity, we restrict ourselves to $w_{d+1}=1$, i.e., we only consider scalar networks.

Given an architecture $\chi$, we may consider the set of admissible weights
\begin{align*}
    \W(\chi):=\set{\WW =(\bW_0,\bB_0,\ldots,\bW_{d},\bB_{d})}{\bW_i\in\R^{w_{i+1}\times w_i}\text{ and }\bB_i\in\R^{w_{i+1}}\text{ for }i=0,\ldots,d}.
\end{align*}

We fix the activation function as (Leaky-)ReLU
\begin{align*}
    \phi(x):=\max\{x,\delta_{\rm ReLU} x\}
\end{align*}
for some fixed $0\leq \delta_{\rm ReLU}<1$. While the precise nature of the activation is not important, we will use that $\phi$ is homogeneous of some degree $k\in\N$ ($k=1$ for Leaky-ReLU).

Given the weights $\WW\in\W(\chi)$, we may define the realization of the neural network as \rev{$\RR(\WW)\colon \R^{w_0}\to \R^{w_{d+1}}$, where} $\RR(x,\WW):=\RR_d(x,\WW)$ and 
\begin{align}\label{eq:nndef}
\begin{split}
 \RR_0(x,\WW)&:= \bW_0 x+ \bB_0,\\
 \RR_{i}(x,\WW)&:= \bW_i \phi( \RR_{i-1}(x,\WW)) + \bB_i\quad\text{for all }i=1,\ldots,d.
 \end{split}
\end{align}
For $\WW\in\W(\chi)$, we define the number of trainable parameters as 
\begin{align*}
\#\WW:=\#\chi:= \sum_{i=0}^d (w_i+1)w_{i+1}.
\end{align*}
This notation allows us to identify $\W(\chi)$ with a subset of $\R^{\#\chi}$ by vectorization of all the elements of the $(2d+2)$-tuples $\WW\in\W(\chi)$.

We define the set of realizations for a given architecture by
\begin{align*}
    \RR(\chi):= \set{x\mapsto \RR(x,\WW)}{\WW\in\W(\chi)}
\end{align*}

\subsection{Loss function and training data}
We assume a set of training data $x_1,\ldots,x_n\in\R^{w_0}$ and responses $y_1,\ldots,y_n\in\R$ to be given.
For $F\in\RR(\chi)$ with $F=\RR(\WW)$, we define the vectors
\begin{align*}
\bRR(\WW):=\bF:=\big(F(x_i)\big)_{i=1}^n\in\R^n.
\end{align*} With $\by:=(y_i)_{i=1}^n$, we are interested in minimizing the loss
\begin{align}\label{eq:loss}
    \LL(F)\rev{:=\LL(\WW)}:= |\by-\rev{{\mathrm H}}\bF|^2,
\end{align}
for a given (semi) norm $|\cdot|$ with corresponding inner product $\cdot$ on $\R^n$ (note that it is essential that the norm is related to an inner product) and a linear operator ${\mathrm H}$ that maps the set of all possible realizations into the space of functions $\R^{w_d}\to \R$. Thus, ${\mathrm H}F(x_i)$ is understood in the sense ${\mathrm H}(x\mapsto F(x))(x_i)$ and correspondingly $(\rev{{\mathrm H}}\bF)_i:= (HF)(x_i)$ with $\rev{{\mathrm H}}\bF\in\R^n$.

A common example would be the normalized least-squares loss corresponding to the norm 
\begin{align*}
    |\bz|^2 := \frac{1}{n}\sum_{i=1}^n |z_i|^2\quad\text{for all }\bz\in\R^n.
\end{align*}
However, it may also become handy to weight the data points non-uniformly, i.e.,
\begin{align*}
    |\bz|^2 := \sum_{i=1}^n \gamma_i|z_i|^2\quad\text{for all }\bz\in\R^n\quad\text{with }\sum_{i=1}^n \gamma_i=1\text{ and } \gamma_i\geq 0.
\end{align*}
This also includes loss functionals typically used in the context of PINNs, of the form
\begin{align*}
 \int_D |D_x F(x) - y(x)|^2\,dx \approx \LL(F):=\sum_{i=1}^n\gamma_i |(D_xF)_i - y(x_i)|^2,
\end{align*}
where $H=D_x$ is some differential operator on $x$ and $D$ some physical domain. A practical implementation of course has to choose quadrature weights and points $(\gamma_i,x_i)$.

\subsection{Hierarchical algorithm and definition of optimality}
The main idea of the work is to start training a small network and extend it as soon as the optimization algorithm reaches (comes close) to a stationary point and the loss does not improve anymore. This leads to a sequence of architectures $\chi_0,\chi_1,\ldots$ and a sequence of weights $\rev{\WW_\ell}\in\W(\chi_\ell)$ with the goal of reducing the loss
\begin{align*}
    \LL(\rev{\WW_0})>\LL(\rev{\WW_1})>\ldots
\end{align*}
\rev{while controlling the size of the resulting networks $\#\chi_\ell$.
 Ideally, we would want a training algorithm that is optimal in the following sense: Given a loss function $\LL(\cdot)$, if there exists a sequence of networks $G_\ell\in\RR(\chi_\ell^{\rm opt})$ with
\begin{align*}
    \LL(\GG_\ell)\leq C_{\rm opt} (\#\chi_\ell^{\rm opt})^{-1/\beta}\quad\text{for all }\ell\in\N
\end{align*}
for some $C,\beta>0$, then the training algorithm produces weights $\WW_\ell\in\W(\chi_\ell)$ with
\begin{align*}
    \LL(\WW_\ell)\leq C (\#\chi_\ell)^{-1/\beta}\quad\text{for all }\ell\in\N.
\end{align*}
\rev{The rest of the work shows that Algorithm~\ref{alg:adaptive} achieves a weaker form of this ideal optimality (Theorem~\ref{thm:main}). To describe the algorithm, we make the following assumptions on the optimizer ${\rm OPT}_{\lnot \WW'}(\WW)$ used in Steps~(1) \& (2c) of Algorithm~\ref{alg:adaptive}:}
\begin{itemize}
    \item It minimizes $\LL(\cdot)$ by optimizing over the parameters in $\WW$ that are not in $\WW'$ (parameters in $\WW'$ remain unchanged, if $\WW'$ is missing all parameters of $\WW$ are optimized).
    \item The output is locally minimal under scaling, i.e., $\widehat\WW = {\rm OPT}_{\lnot\WW'}(\WW'\oplus \WW'')$ has the form $\widehat \WW = \WW'\oplus \widehat\WW''$ and $\partial_\alpha \LL(\WW'\oplus \alpha\widehat \WW'')= 0$. This can be achieved by solving a one-dimensional quadratic minimization problem in post processing.
    \item The output of ${\rm OPT}$ has loss not greater than its input. This can always be achieved by backtracking.
\end{itemize}}
Note that we do not assume that the optimization algorithm in Steps~(1) or~(2c) of Algorithm~\ref{alg:adaptive} finds an actual minimizer of the loss.

\begin{algorithm}\caption{\rev{Hierarchical Algorithm}}\label{alg:adaptive}
 \rev{Input: Weights $\WW_0\in\W(\chi_0)$, architecture $\chi^\star$, parameter $\frac{\sqrt{5}-1}{2}<\kappa<1$.\\
 For $\ell=0,1,2,\ldots$ do:
 \begin{enumerate}
  \item Optimize $\WW_{\ell,0}={\rm OPT}(\WW_\ell)$.
  \item For $j=0,1,2,\ldots$ do:
 \begin{enumerate}
  \item Compute 
  \begin{align}\label{eq:Wstar}
 \WW^\star := {\rm arg}\max_{\WW\in \W(
 \chi^{\star})\atop |\rev{{\mathrm H}}\bRR(\WW)| \leq 1} (\by-\rev{{\mathrm H}}\bRR(\WW_{\ell,j}))\cdot \rev{{\mathrm H}}\bRR(\WW).
\end{align}
  \item Update $\WW_{\ell,j+1} = \WW_{\ell,j} \oplus \WW^\star$.
  \item Optimize $\WW_{\ell,j+1}={\rm OPT}_{\lnot\WW_{\ell,0}}(\WW_{\ell,j+1})$ 
  \item If $\LL(\WW_{\ell,j+1})\leq \kappa \LL(\WW_{\ell,0})$, set $\WW_{\ell+1}=W_{\ell,j+1}$ and exit inner loop.
 \end{enumerate}
 \end{enumerate}
 Output: sequence of weights $(\WW_\ell)_{\ell\in\N}$.}
\end{algorithm}
\begin{remark}\label{rem:Wstar}
    While the optimization problem~\eqref{eq:Wstar} is still non-convex and thus non-trivial in general, it is posed on a much smaller search space $\W(\chi^\star)$ instead of $\W(\chi)$. If $\#\chi^\star$ is small, an exhaustive search via sampling is feasible.  In some cases,~\eqref{eq:Wstar} can be reduced to a convex constrained optimization problem, for which fast algorithms with convergence guarantees exist. 
 There is extensive literature about rewriting neural network training into convex programs, see e.g.~\cite{convexoptNN,convexopt2}. 
    We note that weights $\WW^\star$ that maximize~\eqref{eq:Wstar} only up to a multiplicative constant are sufficient for the results of this work. Moreover,~\eqref{eq:Wstar} can be optimized in parallel, i.e., by starting many small parallel optimization processes from properly distributed starting points. This is not the case for standard optimization in $\W(\chi)$, where the curse of dimensionality prevents exhaustive search and standard optimization steps have to be done sequentially.
\end{remark}

\begin{remark}
    Instead of~\eqref{eq:Wstar}, one could alternatively try to solve
     \rev{\begin{align}\label{eq:WstarDirect}
 \WW^{\star\star} = {\rm arg}\min_{\WW\in \W(
 \chi^{\star})\atop |H\RR(\WW)| \leq 1} |\by-\rev{{\mathrm H}}\bRR(\WW_{\ell,j+1})|^2,
\end{align}}
where $\WW_{\ell,j+1}$ is updated with $\WW^{\star\star}$ instead of $\WW^\star$. We observed in the numerical examples, that this is less effective than using~\eqref{eq:Wstar}.
\end{remark}

\subsection{Neural network calculus}
\rev{We define $\alpha \WW$ such that $\RR(\alpha\WW)=\alpha\RR(\WW)$ for all $\alpha\geq 0$. This can, e.g., be achieved by scaling the final layer, i.e., $\alpha\bW_d$ and $\alpha\bB_d$.} Given $F\in\RR(\chi)$, we know that there exists at least one set of weights $\WW\in\W(\chi)$ such that $F=\RR(\WW)$. Hence, we have the identities
\begin{align*}
    \alpha F= \RR(\alpha\WW)\quad\text{and}\quad \alpha\bF = \bRR(\alpha\WW)\in \R^n.
\end{align*}
We will also require the addition of two realizations of neural networks: Given two architectures $\chi$ and $\chi'$ with $d=d'$ and $w_0=w_0'$, we define the addition $\WW\oplus \WW'$ of $\WW\in\W(\chi)$ and $\WW'\in\W(\chi')$ by
\begin{align*}
    \bW_0\oplus\bW_0' := \begin{pmatrix}
        \bW_0\\ \bW_0'
    \end{pmatrix}
    \quad\text{and}\quad
    \bB_0\oplus\bB_0':= \begin{pmatrix}
        \bB_0\\ \bB_0'
    \end{pmatrix}
\end{align*}
as well as for $i=1,\ldots,d-1$
\begin{align*}
    \bW_i\oplus\bW_i' := \begin{pmatrix}
        \bW_i & 0 \\ 0 &\bW_i'
    \end{pmatrix}
    \quad\text{and}\quad
    \bB_i\oplus\bB_i':= \begin{pmatrix}
        \bB_i\\ \bB_i'
    \end{pmatrix}
\end{align*}
and 
\begin{align*}
    \bW_d\oplus\bW_d' := \begin{pmatrix}
        \bW_d & \bW_d'
    \end{pmatrix}
    \quad\text{and}\quad
    \bB_d\oplus\bB_d':= \bB_d+ \bB_d'.
\end{align*}
This achieves $\RR(\WW\oplus\WW') = \RR(\WW)+\RR(\WW')$.


We start with an elementary observation.
\begin{lemma}\label{lem:descentexists}
Let $\chi,\chi'$ denote architectures, let $F\in\RR(\chi)$, and assume that there exists $G\in\RR(\chi')$ with $\LL(F+G)<\LL(F)$. Then, there holds
 \begin{align}
 \label{eq:descentexists_1}
  \LL(F+ {\alpha G}) \leq \LL(F) -\alpha |\LL(F+G)-\LL(F)|\quad\text{for all }0\leq \alpha\leq 1.
 \end{align}
 Moreover, there holds
 \begin{align*}
 2(\by-\rev{{\mathrm H}}\bF)\cdot \frac{\rev{{\mathrm H}}\bG}{|\rev{{\mathrm H}}\bG|}\geq \frac{|\LL(F+G)-\LL(F)|}{2\sqrt{\LL(F)}}.
 \end{align*}
\end{lemma}
\begin{proof}
 Convexity of the loss functional implies
 \begin{align*}
\LL(F+ \alpha G) &=\big|(1-\alpha)(\by-\rev{{\mathrm H}}\bF) + \alpha\big(\by-\rev{{\mathrm H}}\bF-\rev{{\mathrm H}}\bG\big)\big|^2\\
&\leq (1-\alpha)\LL(F) + \alpha\LL(F+G).
\end{align*}
This concludes proof of the first statement.

For the second statement, we first observe that
\begin{align*}
 |\rev{{\mathrm H}}\bG|\leq \sqrt{\LL(F)}+\sqrt{\LL(F+G)}\leq 2\sqrt{\LL(F)}
\end{align*}
as well as
\begin{align}\label{eq:chainrule}
\begin{split}
     \LL(F+\alpha G) &= |\by-\rev{{\mathrm H}}\bF-\alpha\rev{{\mathrm H}}\bG|^2 = \LL(F) + |\alpha \rev{{\mathrm H}}\bG|^2 - 2(\by-\rev{{\mathrm H}}\bF)\cdot \alpha \rev{{\mathrm H}}\bG\\
     &=
       \LL(F) + \alpha^2|\rev{{\mathrm H}}\bG|^2 - 2\alpha(\by-\rev{{\mathrm H}}\bF)\cdot \rev{{\mathrm H}}\bG.   
       \end{split}
 \end{align}
Using this and~\eqref{eq:descentexists_1}, we show
\begin{align*}
 -\alpha |\LL(F+G)-\LL(F)|&\geq \LL(F+ \alpha G) - \LL(F) \\
 &= -2\alpha(\by-\rev{{\mathrm H}}\bF)\cdot \rev{{\mathrm H}}\bG+ \alpha^2|\rev{{\mathrm H}}\bG|^2.
\end{align*}
Thus, passing to the limit $\alpha\to0$, we conclude the proof of the second statement. 
 \end{proof}

\begin{figure}
\begin{center}
\tdplotsetmaincoords{70}{60} 

\begin{tikzpicture}[tdplot_main_coords,scale=1.4]
    \coordinate (O) at (0,0,0);
    \coordinate (I) at (0,0,0);
    \filldraw (I) circle (2pt); 

    \draw[->,red!50] (O) -- (2,0,0) node[anchor=north east]{};
    \draw[->,red!50] (O) -- (0,2,0) node[anchor=north west]{};
    \draw[->,red!50] (O) -- (0,0,2) node[anchor=south]{};

    \foreach \x in {-0.5,0.5}
        \foreach \y in {-0.5,0.5}
            \foreach \z in {-0.5,0.5}{
                \draw[->] (I) -- (\x,\y,\z); 
            }

\end{tikzpicture}
\qquad\qquad
\begin{tikzpicture}[tdplot_main_coords,scale=1.4]
    \coordinate (O) at (0,0,0);
    \coordinate (I) at (0,0,0);
    \filldraw (I) circle (2pt); 

    \draw[->,red!50] (O) -- (2,0,0) node[anchor=north east]{};
    \draw[->,red!50] (O) -- (0,2,0) node[anchor=north west]{};
    \draw[->,red!50] (O) -- (0,0,2) node[anchor=south]{};

    \foreach \x in {0.2,1}
        \foreach \y in {0.3,1}
            \foreach \z in {0.1,1}{
                \draw[->] (I) -- (\x,\y,\z); 
            }

\end{tikzpicture}
\end{center}
\caption{Schematic of the stability assumption in~\eqref{eq:stability0} and Definition~\ref{def:Lstable}. The black vectors represent the $\bz_i$. While the configuration on the right-hand side is okay, the left-hand side configuration violates~\eqref{eq:stability0} due to cancellation.}
\label{fig:stability}
\end{figure}

\section{Quasi-optimal hierarchical training}
\rev{Step~(2) of Algorithm~\ref{alg:adaptive} searches for a good extension $\WW_{\ell+1}$ such that the loss is minimized. We will derive the existence of such an extension from the assumption that there exists a (theoretical) neural network $G\in\RR(\chi)$ that satisfies
\begin{align}\label{eq:firstimprov}
    \LL(\RR(\WW_\ell)+G)\ll \LL(\RR(\WW_\ell)).
\end{align}
\rev{(Note that $\RR(\WW_\ell)+G$ is understood as the addition of two functions from $\R^{w_0}$ to $\R^{w_{d+1}}$.)}
Considering the myriad of neural network approximation results (see, e.g.,~\cite{schwab_approx,schwab_approx2,schwab_approx3,jentzen_approx}), the existence of $G$ seems to be a reasonable assumption.
However, the architecture $\chi$ could be prohibitively large and extending the architecture in that fashion would result in $\#\chi_{\ell+1}\gtrsim \#\chi_\ell + \#\chi$ and thus a prohibitively large search space in~\eqref{eq:Wstar} to find the initialization of the new weights in $\WW_{\ell+1}$.}

We require the following assumption that allows us to restrict the optimization problem~\eqref{eq:Wstar} to a small search space.
\begin{assumption}\label{ass:implication}
There exists a constant $L>0$ such that the existence of architectures $\chi,\chi^\star$ (where $\chi^\star$ is assumed to be much smaller than $\chi$), as well as $G\in\RR(\chi)$, $\gamma>0$ with
\begin{subequations}\label{eq:implication}
\begin{align}\label{eq:implication0}
\frac{\bz \cdot \rev{{\mathrm H}}\bG}{|\bz||\rev{{\mathrm H}}\bG|}\geq \gamma>0
\end{align}
for any vector $\bz\in \R^n$ implies that there exists $\WW^\star\in\chi^\star$ such that
\begin{align}\label{eq:implication1}
    \frac{\bz\cdot \rev{{\mathrm H}}\bRR(\WW^\star)}{|\bz||\rev{{\mathrm H}}\bRR(\WW^\star)|}\geq \frac{\gamma}{L\sqrt{\#\chi/\#\chi^\star}}. 
\end{align}
\end{subequations}
\end{assumption}
Note that the validity of~Assumption~\ref{ass:implication} is unclear in general. However, we will see in the following Section~\ref{sec:stability}, that numerical stability of $G$ implies~\eqref{eq:implication} and thus enables a rigorous proof of our main result in Theorem~\ref{thm:main}.

\subsection{Validity of Assumption~\ref{ass:implication} and numerical stability of neural networks}\label{sec:stability}
We start with a geometric observation that implies the stability of sums of vectors. As illustrated in Figure~\ref{fig:stability}, if the vectors $\bz_i$ in Lemma~\ref{lem:independence} below {point into similar directions,} at least one of the vectors must be roughly parallel to their sum. The fact is quantified by the constant $L$ in~\eqref{eq:stability0}.
\begin{lemma}\label{lem:independence}
 Let $\bz_1,\ldots,\bz_w\in\R^n$ such that
\rev{\begin{align}\label{eq:stability0}
|\sum_{j=1}^w \alpha_j\bz_j|\leq  L\max_{j=1,\ldots,w}|\alpha_j|\big|\sum_{j=1}^{w}\bz_j\big|
\end{align}
for all $\alpha_1,\ldots,\alpha_w\in \R$ and some $L>0$.}
Then, there holds for all $\by\in\R^n$ that
 \begin{align}\label{eq:stability}
   \frac{|\by \cdot \sum_{j=1}^w \bz_j|}{|\sum_{j=1}^w \bz_j|}\leq 
    Lw^{1/2}\max_{j=1,\ldots,w}\frac{|\by \cdot  \bz_j|}{|\bz_j|}.
\end{align}
If $\by\geq 0$ entrywise, there holds even
\begin{align}\label{eq:stabilitypos}
     \frac{|\by \cdot \phi(\sum_{j=1}^w \bz_j)|}{|\sum_{j=1}^w \bz_j|}\leq 
    Lw^{1/2}\max_{j=1,\ldots,w}\frac{|\by \cdot  \phi(\bz_j)|}{|\bz_j|}.
\end{align}
\end{lemma}
\begin{proof}
\rev{With $\alpha_j =\pm 1$, we get from~\eqref{eq:stability0} that
\begin{align*}
    L^2|\sum_{j=1}^{w}\bz_j|^2\geq \max_{\boldsymbol{\alpha}\in\{-1,1\}^w}|\sum_{j=1}^w \alpha_j\bz_j|^2\geq \E_{\boldsymbol{\alpha}}|\sum_{j=1}^w \alpha_j\bz_j|^2=
    \E_{\boldsymbol{\alpha}}\sum_{i,j=1}^w \alpha_i\alpha_j\bz_j\cdot \bz_j = \sum_{j=1}^w|\bz_j|^2
\end{align*}
where $\E_{\boldsymbol{\alpha}} := 2^{-w}\sum_{\boldsymbol{\alpha}\in\{-1,1\}^w}$.} From this, we obtain
\begin{align*}
    \sum_{j=1}^{w}|\bz_j|^2 +\sum_{i\neq j}\bz_i\cdot\bz_j \geq \frac1{L^2} \sum_{j=1}^{w}|\bz_j|^2
\end{align*}
and hence
\begin{align}\label{eq:Lbound}
    1+\frac{\sum_{i\neq j}\bz_i\cdot\bz_j }{\sum_{j=1}^{w}|\bz_j|^2}\geq \frac{1}{L^2}.
\end{align}
On the other hand, by triangle inequality in the numerator, there holds
\begin{align*}
    \frac{|\by \cdot \sum_{j=1}^w \bz_j|}{|\sum_{j=1}^w \bz_j|}&\leq \frac{\sum_{j=1}^w |\by\cdot \bz_j|}{\sqrt{\sum_{j=1}^{w}|\bz_j|^2 +\sum_{i\neq j}\bz_i\cdot\bz_j }}=
    \frac{\sum_{j=1}^w |\bz_j|\frac{|\by\cdot \bz_j|}{|\bz_j|}}{\sqrt{\sum_{j=1}^{w}|\bz_j|^2 +\sum_{i\neq j}\bz_i\cdot\bz_j }}\\
    &\leq
    \frac{\sqrt{\sum_{j=1}^w |\bz_j|^2}\sqrt{\sum_{j=1}^w\frac{|\by\cdot \bz_j|^2}{|\bz_j|^2}}}{\sqrt{\sum_{j=1}^{w}|\bz_j|^2 +\sum_{i\neq j}\bz_i\cdot\bz_j }}
    = 
    \frac{\sqrt{\sum_{j=1}^w\frac{|\by\cdot \bz_j|^2}{|\bz_j|^2}}}{\sqrt{1+\frac{\sum_{i\neq j}\bz_i\cdot\bz_j }{\sum_{j=1}^{w}|\bz_j|^2}}}
    \leq 
    \frac{w^{1/2}\max_{j=1,\ldots,w}\frac{|\by\cdot \bz_j|}{|\bz_j|}}{\sqrt{1+\frac{\sum_{i\neq j}\bz_i\cdot\bz_j }{\sum_{j=1}^{w}|\bz_j|^2}}}.
\end{align*}
This, together with~\eqref{eq:Lbound}, concludes the proof of the first statement. For the second statement, we use
\begin{align*}
    \phi(a+b)\leq \phi(a) + \phi(b) 
\end{align*}
and follow the arguments above.
\end{proof}
\begin{remark}
    Note that~\eqref{eq:stability0} is always satisfied for some $L$ if $\sum_{j=1}^w \bz_j\neq 0$. In the following, it will be important that $L$ remains reasonably bounded.
\end{remark}
With this geometric stability, we can formulate a notion of stability for neural networks. Heuristically, the definition implies that small changes in the weights of the final layer do not drastically change the output of the network. Such an assumption has to be expected for any numerical approximation algorithm to perform well.
\begin{definition}\label{def:Lstable}
Given an architecture $\chi=(w_0,\ldots,w_{d+1})$,
{ and a tuple of an architecture $\chi^\star$ and parameters $w^\star \in \N$ and $L>0$,
}
a network $\RR(\WW)\in \RR(\chi)$ is called $(L,\chi^\star,w^\star)$-stable 
if $w^\star \#\chi^\star \leq L\#\chi$ and it can be written as
\begin{align}\label{eq:tildeF}
 \RR(x,\WW) = \sum_{j=1}^{w^\star} \RR(x,\WW^{(j)})
\end{align}
for some {weights} $\WW^{(j)}\in\WW(\chi^\star)$ \rev{and it satisfies
\begin{align}
    \Big|{\mathrm H}\bRR(\WW)-\sum_{j=1}^{w^\star}  \alpha_j {\mathrm H}\bRR(\WW^{(j)})\Big|\leq L\max_{j=1,\ldots,w^\star}|1 -\alpha_j||{\mathrm H}\bRR(\WW)|
\end{align}
for all $(\alpha_1,\ldots,\alpha_{w^\star})\in \R^{w^\star}$.}
\end{definition}

\rev{We argue that such numerical stability is necessary for any algorithm to output meaningful results and therefore, any neural network that is useful to perform a given task must be stable in the sense of Definition~\ref{def:Lstable}.}

\rev{How should one choose $\chi^\star$? }
Assuming $H={\rm id}$, one straightforward way to obtain a partition as in~\eqref{eq:tildeF} is to split up the final layer of any given network, i.e., for weights $\WW\in \W(\chi)$, define
\begin{align}\label{eq:starweights}
    \WW^{(j)} := (\bW_0,\bB_0,\ldots,(\bW_{d-1})_{j,:},(\bB_{d-1})_j,(\bW_d)_{1,j},\bB_d/w_d)
\end{align}
with $\WW^{(j)}\in\W(\chi^\star)$ 
for 
\begin{align}\label{eq:chistar}
\chi^{\star}:=(w_0,\ldots,w_{d-1},w_d',1).
\end{align}
for some $w_d'\ll w_d$. Then, $w^\star = w_d/w_d'$. Of course this does not necessarily imply $L$-stability. However, by choosing $w_d'=3$, the final layer of the networks with weights in $\W(\chi^\star)$ are able to express hat functions. For example, given $a<b<c$, and assuming $\delta_{\rm ReLU}=0$ (the ReLU case), the function
\begin{align*}
    F_{a,b,c}(x):=  \begin{pmatrix}
        (b-a)^{-1}, & -  (b-a)^{-1}- (c-b)^{-1},&  (c-b)^{-1}
    \end{pmatrix}\phi\Big(\begin{pmatrix}
        1\\
        1\\
        1
    \end{pmatrix}x + 
    \begin{pmatrix}
    -a\\
    -b\\
    -c
    \end{pmatrix}\Big)
\end{align*}
is a piecewise linear hat function that is zero outside of $(a,c)$ and peaks at $b$ with $F_{a,b,c}(b)=1$. \rev{For networks with $d=2$, the number of linear pieces of $\RR(\WW)$ is bounded in terms of the width of the network, hence $\chi^\star = (w_0,3,1)$ and $w^\star\simeq \max\{w_0,w_1,w_2\}$. In the setting of Lemma~\ref{lem:deepalgoptdir}, we show that it is sufficient to assume that each layer is stable in the sense of Definition~\ref{def:Lstable} and we even obtain that $\chi^\star=(w_0,1,1)$ is sufficient.}

\rev{Since the realizations $\RR(x,\WW^{(j)})$ are piecewise affine functions of $x$, we can further decompose them into sums of networks with hat functions in the final layer. Doing this, we end up with vectors $\bz_j$ in Definition~\ref{def:Lstable} which have bounded overlap, i.e., for any index $i=1,\ldots, n$, only a small number of vectors $\bz_j$ have non-zero entries. This then implies stability in the sense of~\eqref{eq:stability0}.}

Of course the final number of parts $w^\star$ necessary to obtain a good stability estimate can not be bounded in general and will depend on the particular situation. But, in principle, the discussion above shows that $L$-stable networks exist for moderate constants $L$.

\subsection{Assumption~\ref{ass:implication} is true for stable networks}
\rev{The question whether the implication~\eqref{eq:implication} is true is essentially a question about
the minimal necessary complexity of the network architecture $\chi$ to get any meaningful approximation of some target function by some element $G\in\RR(\chi)$.}

\rev{
The following two results for shallow and deep networks show that there is no such complexity barrier if we consider networks that are numerically stable in the sense of Definition~\ref{def:Lstable}. Since unstable networks are very sensitive to tiny perturbations of their weights, we assume that no optimization algorithm can reliable produce them.}

Moreover, the results below can be seen as a counterpart to scaling laws often observed in practical applications of machine learning, see, e.g.,~\cite{scaling,scaling1,scaling2}.

\begin{lemma}\label{lem:algoptdir}
Let $G\in\RR(\chi)$ such that $G$ is $(L,\chi^\star,w^\star)$-stable according to  Definition~\ref{def:Lstable}.
Then, Assumption~\ref{ass:implication} is true.

\end{lemma}
\begin{proof}
We assume~\eqref{eq:implication0}. According to~\eqref{eq:tildeF}, we may rewrite $G$ as
\begin{align*}
 G(x)= \sum_{j=1}^{w^\star}\RR(x,\WW^{(j)})
\end{align*}
and apply Lemma~\ref{lem:independence} to the vectors $\bz_j:=\rev{{\mathrm H}}\bRR(\WW^{(j)})$ to obtain $j_0\in\{1,\ldots,w^\star\}$ with
\begin{align}\label{eq:firstest}
  2\frac{\by-\rev{{\mathrm H}}\bF}{|\by-\rev{{\mathrm H}}\bF|} \cdot \frac{\rev{{\mathrm H}}\bRR(\WW^{(j_0)})}{|\rev{{\mathrm H}}\bRR(\WW^{(j_0)})|}\geq \frac{\gamma}{L\sqrt{w^\star}}.
\end{align}
Since $\WW^{(j_0)}\in \WW(\chi^\star)$ and due to the normalization, we know that the weights $\WW^\star$ from~\eqref{eq:Wstar} satisfy~\eqref{eq:implication1}.
This concludes the proof.
\end{proof}

For deep networks, the following result applies to ordinary ReLU,
$H={\rm id}$.

\rev{
\begin{lemma}\label{lem:deepalgoptdir}
Assume $\delta_{\rm ReLU}=0$ and $H={\rm id}$.
Let $G=\RR(\bW_0,\bB_0,\ldots,\bW_d,\bB_d)\in\RR(\chi)$ and write
\[
 r_i(x):=\RR(x,\bW_0,\bB_0,\ldots,\bW_i,\bB_i),
 \qquad h_i(x):=\phi(r_i(x)),
\]
and denote their sampled scalar components by $
 \boldsymbol r_{i,j}:=(r_i(x_s)_j)_{s=1}^n$ and $\boldsymbol h_{i,j}:=(h_i(x_s)_j)_{s=1}^n$.
In particular, $\boldsymbol r_{d,1}=\bG$.
For $1\leq i\leq d$, consider the scalar layer maps
\[
 F_{i,j}(t):=
 \sum_{k=1}^{w_i}(\bW_i)_{jk}\phi(t_k)+(\bB_i)_j
 \in\RR(\eta_i),
 \qquad \eta_i:=(w_i,w_i,1),
\]
evaluated on the induced data points $t_s^{(i)}:=r_{i-1}(x_s)\in\R^{w_i}$, i.e., $F_{i,j}(t_s^{(i)})=r_i(x_s)_j$.
Assume that, for some $\widetilde L\geq1$, the maps
$F_{i,j}$ are $(\widetilde L,\eta_i^\star,w_i+1)$-stable in the sense of
Definition~\ref{def:Lstable}, where $\eta_i^\star:=(w_i,1,1)$,
with respect to these induced data and the decomposition
\begin{equation}\label{eq:deepstab}
 \widehat F_{i,j}(t)
 =\sum_{k=1}^{w_i+1}\widehat F_{i,j}^{(k)}(t),
 \qquad
 \widehat F_{i,j}^{(k)}(t):=
 \begin{cases}
  c_{i,j}^{-1}(\bW_i)_{jk}\phi(t_k),
      &1\leq k\leq w_i,\\
  c_{i,j}^{-1}(\bB_i)_j,
      &k=w_i+1
 \end{cases}
\end{equation}
with $c_{i,j}:=|\boldsymbol r_{i,j}|$.
For $1\leq i\leq d-1$, set $\kappa_i:=
 \max\left(
 \{1\}\cup
 \left\{
 \frac{|\boldsymbol r_{i,j}|}{|\boldsymbol h_{i,j}|}:
 1\leq j\leq w_{i+1},\
 |\boldsymbol h_{i,j}|>0
 \right\}\right)$,
and define
 $D:=\widetilde L^d
       \prod_{i=1}^d(w_i+1)^{1/2}
       \prod_{i=1}^{d-1}\kappa_i$,
       as well as $\chi^\star:=(w_0,1,1)$.
Then, Assumption~\ref{ass:implication}
holds for entrywise  non-negative vectors $\bz$ with $L=\sqrt{\#\chi^\star/\#\chi}\,D$.
\end{lemma}
}
\begin{proof}
\rev{For each pair $(i,j)$ subject to the stability assumption, set
\[
 \boldsymbol v_{i,j,k}
 :=(\bW_i)_{jk}\boldsymbol h_{i-1,k}
 \quad(1\leq k\leq w_i),
 \qquad
 \boldsymbol v_{i,j,w_i+1}
 :=(\bB_i)_j\boldsymbol1.
\]
These are $c_{i,j}$ times the sample vectors of the summands
in~\eqref{eq:deepstab}, and
\[
 \boldsymbol r_{i,j}
 =\sum_{k=1}^{w_i+1}\boldsymbol v_{i,j,k}.
\]
For signs $\varepsilon_k\in\{-1,1\}$, apply
Definition~\ref{def:Lstable} with $\alpha_k=1-\varepsilon_k$.
Multiplication by $c_{i,j}$ gives $|
 \sum_{k=1}^{w_i+1}\varepsilon_k\boldsymbol v_{i,j,k}
 |
 \leq\widetilde L|\boldsymbol r_{i,j}|$.
As in the proof of~\eqref{lem:independence}, this implies
\begin{equation}\label{eq:deepstabenergy}
 \sum_{k=1}^{w_i+1}|\boldsymbol v_{i,j,k}|^2
 \leq\widetilde L^2|\boldsymbol r_{i,j}|^2.
\end{equation}
For $0\leq i\leq d-1$, define
\[
 M_i:=
 \max\left(
 \left\{\frac{\bz\cdot\boldsymbol1}{|\boldsymbol1|}\right\}
 \cup
 \left\{
 \frac{\bz\cdot\boldsymbol h_{i,j}}{|\boldsymbol h_{i,j}|}:
 1\leq j\leq w_{i+1},\
 |\boldsymbol h_{i,j}|>0
 \right\}\right).
\]
All these quantities are non-negative. Vectors of zero seminorm
have zero inner product with $\bz$ and may be omitted throughout.}

\rev{Each incoming summand is a scalar multiple of a non-negative
activation vector or of $\boldsymbol1$. Therefore,
\begin{equation}\label{eq:deepatom}
 \bz\cdot\phi(\boldsymbol v_{i,j,k})
 \leq M_{i-1}|\boldsymbol v_{i,j,k}|.
\end{equation}
Indeed, for a non-negative scalar coefficient this follows from
the definition of $M_{i-1}$, while for a negative coefficient the
left-hand side is zero.}

\rev{For $1\leq i\leq d-1$ and $j$ with
$|\boldsymbol h_{i,j}|>0$, subadditivity of ReLU, non-negativity
of $\bz$, Cauchy--Schwarz, and
\eqref{eq:deepstabenergy} imply
\begin{align*}
 \bz\cdot\boldsymbol h_{i,j}
 &\leq
 \sum_{k=1}^{w_i+1}
 \bz\cdot\phi(\boldsymbol v_{i,j,k})
 \leq
 M_{i-1}\sum_{k=1}^{w_i+1}|\boldsymbol v_{i,j,k}|
 \leq
 M_{i-1}\sqrt{w_i+1}
 \left(
 \sum_{k=1}^{w_i+1}|\boldsymbol v_{i,j,k}|^2
 \right)^{1/2}\leq
 M_{i-1}\widetilde L\sqrt{w_i+1}
 |\boldsymbol r_{i,j}|.
\end{align*}
Dividing by $|\boldsymbol h_{i,j}|$ yields
\[
 \frac{\bz\cdot\boldsymbol h_{i,j}}{|\boldsymbol h_{i,j}|}
 \leq
 \widetilde L\sqrt{w_i+1}\,\kappa_i M_{i-1}.
\]
The constant term in $M_i$ satisfies the same bound, since it
is already included in $M_{i-1}$ and
$\widetilde L\sqrt{w_i+1}\kappa_i\geq1$. Consequently,
\begin{equation}\label{eq:deeprecursion}
 M_i\leq
 \widetilde L\sqrt{w_i+1}\,\kappa_i M_{i-1}
 \qquad(1\leq i\leq d-1).
\end{equation}
At the output layer, use
$\boldsymbol v\leq\phi(\boldsymbol v)$ entrywise and repeat
the same estimates to obtain
\[
 \bz\cdot\bG
 =\sum_{k=1}^{w_d+1}\bz\cdot\boldsymbol v_{d,1,k}
 \leq
 M_{d-1}\widetilde L\sqrt{w_d+1}\,|\bG|.
\]
Iterating~\eqref{eq:deeprecursion} gives
\[
 \frac{\bz\cdot\bG}{|\bG|}\leq D M_0.
\]}

\rev{The maximum defining $M_0$ is attained either by the constant
function or by a first-layer neuron
\[
 x\mapsto\phi((\bW_0)_{j,:}x+(\bB_0)_j),
\]
after normalization of its sample vector. Each such function
belongs to $\RR((w_0,1,1))$. Taking this normalized function
as $g$ gives
\[
 \frac{\bz\cdot\boldsymbol g}{|\bz|}
 =\frac{M_0}{|\bz|}
 \geq
 \frac{\bz\cdot\bG}{D|\bz||\bG|}
 \geq\frac{\gamma}{D}.
\]
Finally,
$L\sqrt{\#\chi/\#\chi^\star}=D$,
which proves the stated restricted implication.}
\end{proof}

\subsection{Analysis of Step~(2) of Algorithm~\ref{alg:adaptive}}

The first result is an immediate consequence of~\eqref{eq:implication}.
\begin{lemma}\label{lem:reduction}
 Let $F\in\RR(\chi)$ and let $G$ satisfy \rev{Assumption~\ref{ass:implication}} as well as
\begin{align*}
\LL(F+G)   \leq  \kappa \LL(F)
\end{align*}
\rev{for some $0<\kappa<1$.}
\rev{Then, there exist weights $\WW^{\star}\in\W(\chi^\star)$ with}
\begin{align*}
  \rev{\LL(F+ \RR(\WW^\star))}
 \leq  \Big(1-\frac{(1-\kappa)^2}{{8}L^2 \#\chi/\#\chi^\star}\Big)\LL(F).
\end{align*}
\end{lemma}
\begin{proof}
We employ Lemma~\ref{lem:descentexists} to obtain
 \begin{align*}
 2(\by-\rev{{\mathrm H}}\bF)\cdot \frac{\rev{{\mathrm H}}\bG}{|\rev{{\mathrm H}}\bG|}\geq  \frac{1-\kappa}{2}\sqrt{\LL(F)}.
 \end{align*}
 \rev{Thus,~\eqref{eq:implication} shows that there exist  weights $\WW^{\star}$ such that}
 \begin{align*}
2(\by-\rev{{\mathrm H}}\bF)\cdot \rev{{\mathrm H}}\bRR(\WW^{\star})&\geq \frac{1-\kappa}{2L\sqrt{\#\chi/\#\chi^\star}}\sqrt{\LL(F)}.
 \end{align*}
 Thus,~\eqref{eq:chainrule} shows
 \begin{align*}
  \LL(&F+ \alpha\RR(\WW^{\star})) - \LL(F) \\
 &= -2\alpha (\by-\rev{{\mathrm H}}\bF)\cdot \rev{{\mathrm H}}\bRR(\WW^{\star})+ {\alpha^{2}} |\rev{{\mathrm H}}\bRR(\WW^{\star})|^2\\
 &\leq  -\alpha \frac{1-\kappa}{2L\sqrt{\#\chi/\#\chi^\star}}\sqrt{\LL(F)}+ {\alpha^{2}}
 \end{align*}
 With ${\alpha}= \frac{1-\kappa}{{4}L\sqrt{\#\chi/\#\chi^\star}}\sqrt{\LL(F)}$ this implies
\begin{align*}
  \LL(F+ \alpha\RR(\WW^{\rm opt})) - \LL(F) 
 &\leq   -\frac{(1-\kappa)^2}{{8}L^2\#\chi/\#\chi^\star}\LL(F).
 \end{align*}
 \rev{We conclude the proof by replacing $\WW^\star$ with $\alpha\WW^\star$.}
\end{proof}

The following lemma ensures that Step~(2) of Algorithm~\ref{alg:adaptive} does not change the realization of the networks $\RR(\WW_\ell)$ too much unless the loss is significantly improved. This implies a certain stability of the algorithms that will be useful later.
\begin{lemma}\label{lem:intermediateOpt}
\rev{Assume that Step~(2) of Algorithm~\ref{alg:adaptive} runs for $L_{\rm max}$ steps. Then, the output satisfies}
\begin{align*}
    |\rev{{\mathrm H}}\bRR(\WW_{\ell,L_{\rm max}})-\rev{{\mathrm H}}\bRR(\WW_{\ell,0})|^2 {=} \LL(\WW_{\ell,0})-\LL(\WW_{\ell,L_{\rm max}}).
\end{align*}
\end{lemma} 
\begin{proof}
By design of Algorithm~\ref{alg:adaptive}, the weights $\WW_{\ell,L_{\rm max}}$ are the sum of the individual updates $\WW^\star$. Thus, we may write $\WW_{\ell,L_{\rm max}}=\WW_{\ell,0}\oplus\WW_{\ell,\perp}$. By definition of $\oplus$, this implies $\RR(\WW_{\ell,\perp})=\RR(\WW_{\ell,L_{\rm max}})-\RR(\WW_{\ell,0})$. Since $\WW_{\ell,L_{\rm max}}$ is locally minimal under scaling (by assumption on {\rm OPT}), we have
 \rev{\begin{align*}
    0&= \partial_\alpha \LL(\WW_{\ell,0}\oplus \alpha \WW_{\ell,\perp})|_{\alpha=1} =\partial_\alpha 2(\by-\rev{{\mathrm H}}\bRR(\WW_{\ell,L_{\rm max}}))\cdot \rev{{\mathrm H}}\bRR(\alpha\WW_{\ell,\perp})|_{\alpha=1}\\
    &= 2(\by-\rev{{\mathrm H}}\bRR(\WW_{\ell,L_{\rm max}}))\cdot \partial_\alpha(\alpha \rev{{\mathrm H}}\bRR(\WW_{\ell,\perp}))|_{\alpha=1}
    =2(\by-\rev{{\mathrm H}}\bRR(\WW_{\ell,L_{\rm max}}))\cdot (\rev{{\mathrm H}}\bRR(\WW_{\ell,L_{\rm max}})-\rev{{\mathrm H}}\bRR(\WW_{\ell,0}))
 \end{align*}}
 and hence
 \rev{\begin{align*}
     |\rev{{\mathrm H}}\bRR(\WW_{\ell,L_{\rm max}})-\rev{{\mathrm H}}\bRR(\WW_{\ell,0})|^2 =  |\by-\rev{{\mathrm H}}\bRR(\WW_{\ell,0})|^2 - |\by-\rev{{\mathrm H}}\bRR(\WW_{\ell,L_{\rm max}})|^2 =\LL(\WW_{\ell,0})-\LL(\WW_{\ell,L_{\rm max}}). 
 \end{align*}}
This concludes the proof. 
\end{proof}

The following result is the heart of the Theorem~\ref{thm:main} below. It shows that if a good extension of the network exists in theory, Step~(2) of Algorithm~\ref{alg:adaptive} can find a comparably good approximation with a similar number of weights.
\begin{theorem}\label{thm:inner}
Choose $\frac{\sqrt{5}-1}{2}<\kappa<1$, define $0<\mu:=(\kappa-\sqrt{1-\kappa})^2<1$, and let $F:=\RR(\WW_{\ell,0})\in\RR(\chi_0)$. Let $G\in\RR(\chi)$ satisfy \rev{Assumption~\ref{ass:implication}} as well as $\LL(F+G)\leq \mu \LL (F)$. Then, Step~(2) of Algorithm~\ref{alg:adaptive} produces $\WW_{\ell,L_{\rm max}}$ after \rev{$L_{\rm max}=\lceil C_{\rm desc} \#\chi/\#\chi^\star\rceil$} steps with 
$\LL(\WW_{\ell,L_{\rm max}})\leq \kappa \LL(\WW_{\ell,0})$ and $\#\WW_{\ell,L_{\rm max}}-\#\WW_{\ell,0} \leq L C_{\rm desc} \#\chi$, where $C_{\rm desc}:=1+\frac{{8}L^2|\log(\kappa)|}{(1-\kappa)^2}$.
\end{theorem}
\begin{proof}
We use the shorthand $F^{(j)}:=\RR(\WW_{\ell,j})$ and first show 
\begin{align*}
    \sqrt{\LL(F^{(j)}+G)} &= |\by-\rev{{\mathrm H}}\bF^{(j)}-\rev{{\mathrm H}}\bG|
    \leq |\by-\rev{{\mathrm H}}\bF^{(0)}-\rev{{\mathrm H}}\bG| + |\rev{{\mathrm H}}\bF^{(0)}-\rev{{\mathrm H}}\bF^{(j)}|\\
   & \leq \sqrt{\LL(F+G)} + \sqrt{\LL(F)-\LL(F^{(j)})}\\
   &\leq \sqrt{\mu}\sqrt{\LL(F)} + \sqrt{\LL(F)-\LL(F^{(j)})},
\end{align*}
where we used  Lemma~\ref{lem:intermediateOpt} for the second inequality.
If $\LL(F^{(j)})\geq \kappa \LL(F)$, this implies that
\begin{align*}
    \sqrt{\LL(F^{(j)}+G)} & \leq  \sqrt{\mu}\sqrt{\LL(F)}+ \sqrt{(1-\kappa) \LL(F)} = \frac{\sqrt{\mu}+\sqrt{1-\kappa}}{\sqrt{\kappa}}\sqrt{\LL(F^{(j)})}=\sqrt{\kappa}\sqrt{\LL(F^{(j)})}.
\end{align*}
Heuristically, this shows that if Step~(2) of Algorithm~\ref{alg:adaptive} fails to improve the loss by a certain amount in step $j$, we can still add $G$ to $F^{(j)}$ to reduce the loss (just as we can do with $F^{(0)}$, i.e., it does not destroy previous progress. 
Particularly, in each iteration of Step~2 of Algorithm~\ref{alg:adaptive}, we may apply Lemma~\ref{lem:reduction} to obtain
\begin{align*}
    \LL(F^{(\ell+1)})\leq \LL(F^{(j)}+ \RR(\WW^{\rm opt}_\ell))\leq (1-\frac{(1-\kappa)^2}{8L^2\#\chi/\#\chi^\star})\LL(F^{(j)}),
\end{align*}
\rev{where $\WW^{\rm opt}_\ell$ is $\WW^\star$ from Step~(2a) further optimized in Step~(2c). By assumption on ${\rm OPT}$, $\WW_\ell^{\rm opt}$ is locally minimal under scaling, i.e.,
$\LL(F^{(j)}+ \RR(\WW^{\rm opt}_\ell))\leq \LL(F^{(j)}+\alpha \RR(\WW^\star))$.}
Thus, for
\begin{align*}
\ell =  \Big\lceil|\log(\kappa)|\Big/\Big|\log\Big(1-\frac{(1-\kappa)^2}{{8}L^2\#\chi/\#\chi^\star}\Big)\Big|\Big\rceil\leq \Big(1+\frac{{8}L^2|\log(\kappa)|}{(1-\kappa)^2}\Big)\,\#\chi/\#\chi^\star
\end{align*}
this implies
\begin{align*}
    \LL(F^{(j)})\leq \rev{\kappa}\LL(F).
\end{align*}
In each iteration of Step~2 of Algorithm~\ref{alg:adaptive}, $\#\WW_\ell$ grows by $\#\chi^\star$. This concludes the proof.
\end{proof}

The following result shows that Algorithm~\ref{alg:adaptive} produces optimal training results under certain assumptions.
\begin{theorem}\label{thm:main}
\rev{Choose $\kappa$ as in Algorithm~\ref{alg:adaptive} and} $\mu$ as in Theorem~\ref{thm:inner}.
Given $\WW_0\in\W(\chi_0)$ and let $(\WW_\ell)_{\ell\in\N}$ denote the output of Algorithm~\ref{alg:adaptive}. With $F_\ell:=\RR(\WW_\ell)$, assume that for all $k,\ell\in\N$, there exists $G_{k,\ell}\in\RR(\chi_{k,\ell})$ such that $\LL(F_\ell+G_{k,\ell})\leq 2^{-k}\LL(\WW_0)$ and $\#\chi_{k,\ell}\leq C_{\rm opt} 2^{\beta k}+{\delta \#\WW_\ell}$ for all $k\in\N$. 
Moreover, assume that there exists $L>0$ such that all $G_{k,\ell}$ satisfy Assumption~\ref{ass:implication}. Then, there holds
\begin{align*}
    \LL(\WW_\ell)\leq
    L^{1/\beta} C_{\rm desc}^{1/\beta} C_{\rm opt}^{1/\beta}  (2\LL(\WW_0)/\mu) C_{\kappa,\delta,\beta} \#\WW_\ell^{-1/\widetilde\beta},
\end{align*}
for a constant $C_{\kappa,\delta,\beta}>0$ that depends only on $\kappa$, $\beta$, and $\delta$ as well as
\begin{align}\label{eq:wbeta}
    \widetilde \beta :=\max\{\beta, \frac{\log(1+LC_{\rm desc}\delta)+\delta}{|\log(\kappa)|}\}.
\end{align}
Note that sufficiently small $\delta>0$ results in $\widetilde \beta=\beta$ and large $\delta$ results in a reduced rate of convergence $1/\widetilde\beta<1/\beta$.
\end{theorem}

\begin{proof}[Proof of Theorem~\ref{thm:main}]
{We fix $\ell \in \N$. Since we assumed that
the optimizer does not increase the loss, i.e., $\LL(\WW_{\ell}) \leq \LL(\WW_{0})$
and $\mu \in (0,1)$, we can} choose $k\in\N$ such that $\LL(\WW_\ell)/2< 2^{-k}\LL(\WW_0)/\mu\leq \LL(\WW_\ell)$ and hence
$\LL(F_\ell+G_{k,\ell})\leq  2^{-k} \LL(\WW_0)\leq \mu \LL(\WW_\ell)$.
Application of Theorem~\ref{thm:inner} with $G=G_{k,\ell}$ shows that Step~(2) of Algorithm~\ref{alg:adaptive} produces $\WW_{\ell+1}$ such that 
\begin{align}\label{eq:reduction}
    \LL(\WW_{\ell+1})\leq \kappa \LL(\WW_\ell)
\end{align}
as well as
\begin{align*}
    \#\WW_{\ell+1}-\#\WW_\ell \leq LC_{\rm desc} \#\chi_{k,\ell}\leq LC_{\rm desc}C_{\rm opt} 2^{\beta k}+{LC_{\rm desc}\delta\#\WW_\ell}.
 \end{align*}
By assumption, this implies
 \begin{align*}
   \#\WW_{\ell+1}\leq
    LC_{\rm desc}C_{\rm opt} (2\LL(\WW_0)/\mu)^{{\beta}} \LL(\WW_\ell)^{-{\beta}}+{(1+LC_{\rm desc}\delta)\#\WW_\ell}
\end{align*}
for all $\ell\in\N$. Iterative application of this inequality shows
\begin{align*}
     \#\WW_{\ell}&\leq
      LC_{\rm desc}C_{\rm opt} (2\LL(\WW_0)/\mu)^{{\beta}} \#\WW_0 \sum_{k=0}^{\ell-1} \LL(\WW_k)^{-\beta}(1+LC_{\rm desc}\delta)^{\ell-k}\\
      &\leq
      LC_{\rm desc}C_{\rm opt} (2\LL(\WW_0)/\mu)^{{\beta}} \#\WW_0 \sum_{k=0}^{\ell-1} \LL(\WW_k)^{-\widetilde \beta}(1+LC_{\rm desc}\delta)^{\ell-k}
\end{align*}
for all $\widetilde \beta\geq \beta$.
Finally,~\eqref{eq:reduction} shows $\LL(\WW_k)^{-\widetilde\beta}\leq \kappa^{\widetilde\beta(\ell-k)}\LL(\WW_\ell)^{-\widetilde\beta}$ for all $\ell\geq k$ and hence
\begin{align*}
   \#\WW_{\ell}\leq
      LC_{\rm desc}C_{\rm opt} (2\LL(\WW_0)/\mu)^{{\beta}} \#\WW_0 \LL(\WW_{\ell})^{-\widetilde\beta}\sum_{k=0}^{\ell-1} (\kappa^{\widetilde\beta}(1+LC_{\rm desc}\delta))^{\ell-k}.
\end{align*}
Choosing $\widetilde \beta$ sufficiently large such that $\kappa^{\widetilde\beta}(1+LC_{\rm desc}\delta)<1$, the sum on the right-hand side is bounded uniformly in $\ell$ by a geometric series.
This concludes the proof.
\end{proof}

\begin{corollary}\label{cor:main}
Choose $\kappa$ and $\mu$ as in Theorem~\ref{thm:inner}.
Given $\WW_0\in\W(\chi_0)$ and let $(\WW_\ell)_{\ell\in\N}$ denote the output of Algorithm~\ref{alg:adaptive}. Assume that for all $k\in\N$, there exists $\widetilde G_{k}\in\RR(\chi_{k})$ such that $\LL(\widetilde G_{k})\leq 2^{-k}\LL(\WW_0)$ and $\#\chi_{k}\leq C_{\rm opt} 2^{\beta k}$ for all $k\in\N$. 
Moreover, assume that there exists $L>0$ such that all $\widetilde G_{k}$ and $\widetilde G_k-\RR(\WW_\ell)$ satisfy Assumption~\ref{ass:implication} for all $\ell,k\in\N$. Then, there holds
\begin{align*}
    \LL(\WW_\ell)\leq
    L^{1/\beta} C_{\rm desc}^{1/\beta} C_{\rm opt}^{1/\beta}  (2/\mu) C_{\kappa,\delta,\beta} \#\WW_\ell^{-1/\widetilde\beta},
\end{align*}
where $\widetilde \beta$ is defined in~\eqref{eq:wbeta} with $\delta=1$.
\end{corollary}
\begin{proof}
    We define $G_{k,\ell}:=\widetilde G_k - \RR(\WW_\ell)\in \RR(\chi_{k,\ell})$  and note $\#\chi_{k,\ell}\leq C_{\rm opt}2^{\beta k}+\#\WW_\ell$. 
    There holds $\LL(F_\ell + G_{k,\ell})\leq  \LL(G_k)\leq  2^{-k}\LL(\WW_0)$. Thus, we may apply Theorem~\ref{thm:main} with $\delta=1$ and obtain the result.
\end{proof}

\begin{remark}
    Note that the assumption on the existence of networks $G_{k,\ell}$ in Theorem~\ref{thm:main} statement above is quite strong for small $\delta$ and might be violated in practical examples. While many approximation results for neural networks show similar convergence rates (sometimes even exponential), the problem is that the size of $G_{k,\ell}$ only depends very weakly on $F_\ell$ for small $\delta$. Implicitly, this means that $F_\ell$ does not destroy the approximation. If that would be the case, $G_{k,\ell}$ could have to be at least as big as $F_\ell$ in order to undo the approximation errors of $F_\ell$ before it can improve the approximation. Corollary~\ref{cor:main} does not need that assumption, but might only prove a reduced rate of convergence.
\end{remark}

\subsection{A computable bound on the optimality of the loss}
Given an architecture $\chi$, there exists a (possibly non-unique) set of weights $\WW\in\W(\chi)$ that minimizes the loss $\LL(\RR(\WW))$. In practice, this optimum is unknown and only asymptotic estimates can be applied. 
In general, it is very hard to decide whether the optimization algorithm is stuck in a non-optimal stationary point, or if the network architecture is just too small to give a good approximation.

Under the previous assumption that we can solve the optimization problem~\eqref{eq:Wstar} (almost) exactly (see Remark~\ref{rem:Wstar}), we can
indeed compute a threshold value that partially answers the question above.

\begin{theorem}\label{thm:cea}
We fix an architecture $\chi^\star$ and compute $\WW^\star$ from~\eqref{eq:Wstar} to define the computable constant
\begin{align*}
    C_{\rm opt}:=2(\by-\rev{{\mathrm H}}\bF)\cdot \rev{{\mathrm H}}\bRR(\WW^\star)/\sqrt{\LL(F)}.
\end{align*}
Then, there holds 
\begin{align*}
     \LL(F)\leq  2\LL(F + G)
\end{align*}
for all architectures $\chi$ and networks $G\in\RR(\chi)$ that satisfy Assumption~\ref{ass:implication} with $\#\chi/\#\chi^\star\leq \frac{1}{{16}L^2 C_{\rm opt}^2}$.
\end{theorem}
\begin{remark}
    The statement can be used in the following way: If $C_{\rm opt}$ is small, then even a large extension of the network (quantified by $w^\star$) at hand with some network $G$ can not improve the loss by more than a factor of two. In that case, it might be wiser to train the existing network without increasing its size. If $C_{\rm opt}$ is large, a moderate extension can already improve the loss and hence extension of the network might pay off. One can view Lemma~\ref{thm:cea} as an a~posteriori version of the classical C\'ea lemma for Galerkin methods. Indeed, it quantifies the ratio of current approximation quality to the best possible approximation in a manifold of a given complexity.
\end{remark}
\begin{proof}[Proof of Lemma~\ref{thm:cea}]
Assume there exists $G$ that satisfies Assumption~\ref{ass:implication} such that $\LL(F+G)= \kappa \LL(F)$ for some $0<\kappa<1$. Lemma~\ref{lem:descentexists} shows
 \begin{align*}
 2(\by-\rev{{\mathrm H}}\bF)\cdot \frac{\rev{{\mathrm H}}\bG}{|\rev{{\mathrm H}}\bG|}\geq \frac{|\LL(F)-\LL(F+G)|}{2\sqrt{\LL(F)}}\geq \frac{1-\kappa}{2}\sqrt{\LL(F)}.
 \end{align*}
 Thus, Lemma~\ref{lem:algoptdir} shows that the maximizer $\WW^\star$ satisfies
 \begin{align*}
2(\by-\rev{{\mathrm H}}\bF)\cdot \rev{{\mathrm H}}\bRR(\WW^\star)&\geq \frac{1-\kappa}{{2}L\sqrt{\#\chi/\#\chi^\star}}\sqrt{\LL(F)}.
 \end{align*}
 Assuming that we compute the optimal constant $C_{\rm opt}$, we have
 \begin{align*}
     C_{\rm opt}\geq  \frac{1-\kappa}{{2}L\sqrt{\#\chi/\#\chi^\star}} 
 \end{align*}
 and hence $\kappa\geq 1-{2}C_{\rm opt}L\sqrt{\#\chi/\#\chi^\star}$. This concludes that
 \begin{align*}
    \LL(F)\leq  (1-{2}C_{\rm opt}L\sqrt{\#\chi/\#\chi^\star})^{-1}\LL(F+G)
 \end{align*}
and proves the assertion under the assumption on $\#\chi/\#\chi^\star$.
\end{proof}

\section{Generalization error}

In this section, we assume that $\LL(F)=\frac{1}{n}\sum_{j=1}^n|y_i-F(x_i)|^2$ for simplicity and let $H={\rm id}$.
We assume a data space $\XX\subseteq {\R}^{w_0}$ with a given probability measure $\mu$ on $\XX$ that encodes the data density and a response function $y\colon \XX\to\R$. For a network $F=\RR(\WW)\in\W(\chi)$ trained on a set of data points $\bx=(x_1,\ldots,x_n)\in \XX^n$ and $y_i:=y(x_i)$, $i=1,\ldots,n$ , the generalization error is usually defined by
\begin{align*}
    \GG(F)=\GG(\WW):=\int_{\XX} |F(x)-y(x)|^2\,d\mu(x).
\end{align*}
One is usually interested in the ratio of the (empirical) loss $\LL$ and the (unknown) generalization error $\GG(F)$, i.e., one wants to bound
\begin{align*}
    \GG(F)/\LL(F).
\end{align*}
The usual strategy to 
bound the generalization error $\GG$ is to exploit some topological properties of the training data, e.g., covering numbers or other properties of randomly chosen training data (see, e.g,~\cite{bgj_generalization, deepO_gen, deeplearning, found} for examples).
These approaches usually aim to bound $\GG$ by using some sort of Lipschitz continuity of the network $F$ in a suitable norm $\norm{\cdot}{}$ and the fact that the training data is well-distributed in $\XX$, i.e, for any point $x\in\XX$, there is training data $x_i$ with $\norm{x-x_i}{}$ sufficiently small. Often, this assumption is made in a probabilistic sense, where one assumes a certain distribution of randomly chosen $x_i$. {This is often done when the Rademacher complexity of a hypotheses set is used to bound the generalization error, see, e.g.,~\cite{adanet,deepboost} or~\cite{MLbook} for an overview. Here, the data is i.i.d. distributed and one usually gets a factor $n^{-1/2}$ in the generalization error (akin to a Monte Carlo quadrature error).}
In the present work, however, we will stick with deterministic data. We show that the hierarchical training implies improved generalization error.

\begin{lemma}\label{lem:generr}
    Under the assumptions of Theorem~\ref{thm:main} assume additionally that the produced weights in $\WW_\ell$ have uniformly bounded magnitude and that the response function $y$ is Lipschitz continuous.
    If $\LL(\RR(\WW_\ell))\simeq \eps$ for some $\eps>0$ and {and all sets} $\XX_i:=\set{x\in\XX}{\norm{x-x_i}{}\leq \rev{\eps^{\frac12+d\widetilde\beta}}}$ satisfy $\mu(\XX_i)\simeq 1/n$ {as well as $\bigcup_{i=1}^n\XX_i=\XX$,} then also
    \begin{align*}
        \GG(\RR(\WW_\ell))\leq C\eps.
    \end{align*}
\end{lemma}
\begin{remark}
    Note that for a general training algorithm and without making additional assumptions on the Lipschitz continuity of the trained network, we can not expect a result similar to that of Lemma~\ref{lem:generr} with a constant $C$ that is independent of $\#\WW_\ell$. This follows from the fact that, even with weights that have uniformly bounded magnitude, the Lipschitz constant of a network can grow with $\#\WW_\ell$. (Note that there is strong empirical evidence that the Lipschitz constants of large networks remain bounded or grow very slowly during training)
\end{remark}
\begin{proof}[Proof of Lemma~\ref{lem:generr}]
    Assume that the weights $\bW_i$ and biases $\bB_i$ in $\WW_\ell$ have magnitude bounded by some $C>0$.  A straightforward calculation shows that $\RR(\WW_\ell)$ is Lipschitz continuous, with a Lipschitz constant $C_L\lesssim \#\WW_\ell^{d}$ (for fixed $C$). 
    With the Lipschitz constant $C_L^y>0$ of $y$, there holds for all $x\in\XX_i$
    \begin{align*}
        |\RR({x,\WW_\ell}) - \RR({x_i,\WW_\ell})|+|y(x)-y(x_i)|\leq(C_L^y+C_L)\eps^{1/2+d\widetilde \beta}
    \end{align*}
    This implies
    \begin{align*}
        \GG(\RR(\WW_\ell))&\lesssim\sum_{i=1}^n \mu(\XX_i)\sup_{x\in\XX_i}\Big( |\RR({x,\WW_\ell}) - \RR({x_i,\WW_\ell})|^2+|y(x)-y(x_i)|^2\Big) +\LL({\mathcal{R}(\WW_\ell)})     \\
        &\lesssim (1+\#\WW_\ell^{\rev{2d}}) \eps^{1+2d\widetilde\beta} +\eps.
    \end{align*}
    Theorem~\ref{thm:main} shows $\#\WW_\ell\lesssim \LL(\WW_\ell)^{-\widetilde\beta}\simeq \eps^{-\widetilde\beta}$ and hence
    \begin{align*}
    \GG(\RR(\WW_\ell))\lesssim \eps.
    \end{align*}
    This concludes the proof.
\end{proof}
\subsection{Optimal generalization}
    The proof strategies to bound the generalization error found in the literature (see, e.g.,~\cite{bgj_generalization, deepO_gen}) and employed in the previous section are somewhat unsatisfying as the assumption that the training data is close (in some norm $\norm{\cdot}{}$) to any possible input $x\in\XX$ seems not to hold in practical examples of machine learning. For example, consider a \emph{Large Language Model} $F$ (see, e.g.,~\cite{llm}) that is trained to predict the following sentence of an input text $x$. Experiment shows~\cite{gpt4} that such models often learn other tasks, such as translation, as a byproduct of their training. Thus, the model will likely be able to correctly predict the ending of a translated text $x\in\XX$, even though only the original version of the text $x_i$ is contained in the training data.
In this case, it can not be expected that $\norm{x-x_i}{}$ is small in any norm in which the network $F$ is continuous, as in this case, the norm $\norm{\cdot}{}$ and thus generalization error $\GG$ would not capture if the response $F(x)$ is provided in the language of the input $x$. Moreover, it can not be assumed that the training data is i.i.d in the set of all input data. 

\medskip

To overcome this drawback, we define in Definition~\ref{def:optgen} below the notion of \emph{optimal generalization}. We define a property of the training data that is strictly necessary to be able to achieve small generalization error $\GG$ for general responses $y$. Then, Algorithms with \emph{optimal generalization} that are trained on such training data will also have small generalization error.

In order to make this rigorous, we need to consider a class of training algorithms: We assume a training algorithm $\A$ that takes a given set of training data $x_1,\ldots,x_n$ and the corresponding responses $y_1,\ldots,y_n$ and outputs $(\chi,F):=\A(\bx,\by,\eps)$ with $F\in\W(\chi)$ and $\LL(F)\leq \eps$. We assume that $\A$ is stable in the sense that for different responses $\widetilde y_1,\ldots,\widetilde y_n$ the output of $\A$ satisfies
 \begin{align}\label{eq:stabletraining}
        \int_{\XX} |F(x)-\widetilde F(x)|^2\,d\mu(x)\leq C_{\rm stab}|\by-\widetilde \by|
    \end{align}
for $(\widetilde \chi,\widetilde F)=\A(\bx,\widetilde \by,\eps)$ and a universal constant $C_{\rm stab}>0$.

With this, we define the notion of \emph{optimal generalization} as follows:
We fix a training algorithm $\A$ and consider all possible architectures $\chi$ that can be the output of $\A$.
We call those architectures admissible.
Given $\eps>0$, we denote by $\#\chi(\eps)\in\N$ the smallest number of parameters such that there exists an admissible architecture $
\chi$ with $\#\chi\leq \#\chi(\eps)$ and
\begin{align*}
   \inf_{F\in\W(\chi)} \GG(F)\leq \eps.
\end{align*}
We consider a nested sequence of training data sets $(\bx_k)_{k\in\N}$, i.e., $\bx_k=(x_1,\ldots,x_{n_k})$ for numbers $n_k<n_{k+1}\in\N$.
\begin{definition}\label{def:epsrich}
For fixed responses $y$, we call the sequence $(\bx_k)_{k\in\N}$ of training data $\eps$-rich if for all $C_1>0$ there exists $C_2>0$ such that for all $\eps>0$ and all admissible architectures $\chi$ with $\#\chi\leq C_1\#\chi(\eps) $ the following implication is true for almost all $k\in\N$
\begin{align*}
    F\in\RR(\chi)\text{ with } \LL(F,\bx_k)\leq \eps\quad\implies \quad  \GG(F)\leq C_2\eps,
\end{align*}
where $\LL(F,\bx_k):=\LL(F)=|(y(x_i))_{i=1}^{n_k}-(F(x_i))_{i=1}^{n_k}|^2$ is the loss function defined in~\eqref{eq:loss}. 
\end{definition}
One interpretation of $\eps$-richness of training data is the following: If the smallest architecture required to reduce the (empirical) loss $\LL$ below $\eps$ is of comparable size to the smallest architecture to reduce the generalization error $\GG(\cdot)$ below $\eps$, the empirical loss is a good predictor for the generalization error. We argue that this notion is a necessary requirement for the data in order to meaningfully talk about generalization error, as shown in the following lemma. 
Note that Lemma~\ref{lem:generr} shows that if the sequence of training data is sufficiently well distributed, this implies that the training data is $\eps$-rich (we don't even require the assumption $\#\chi\leq C_1\#\chi(\eps)$ since we use very strong assumptions on the distribution of the data).

\begin{lemma}\label{lem:epsnec}
    Assume a stable training algorithm $\A$ and a sequence of training data $(\bx_k)_{k\in\N}$ that is not $\eps$-rich. Then, for all $C>0$, there exists  $\eps>0$, a response function $\widetilde y$, a set of training data $\bx_k=(x_1,\ldots,x_n)$, $\widetilde \by=(\widetilde y_1,\ldots,\widetilde y_n)$ with $\widetilde y_i:=\widetilde y(x_i)$ such that $(\chi,F)=\A(\bx_k,\widetilde \by,\eps)$ satisfies 
    \begin{align*}
        \widetilde \GG(F):=\int_\XX |F(x)-\widetilde y(x)|^2\,d\mu(x) \geq C \widetilde \LL(F):= C\frac1n\sum_{i=1}^n|\widetilde y_i-F(x_i)|^2.
    \end{align*}
    The response $\widetilde y$ is either equal to $y$ or can be exactly represented by a neural network of a similar size as $\chi$.
\end{lemma}
\begin{remark}
    The result can be interpreted in the following way: If the training set is not $\eps$-rich, a training algorithm that is stable can not have small generalization error, even if we assume that the response function $y$ is exactly represented by a neural network of similar size as the trained network. The fact that $y$ is represented by a network of similar size is important since one always can construct a function $y$ that results in zero loss but arbitrary large $\GG$. If such a function $y$ can be represented by a network of similar size, however, then one can not expect small generalization error without posing additional restrictions or assumption on the architecture of the network. We also refer to~\cite{t2p-gap} for results that show that, for a similar reason, the generalization error can not be expected to be small if one considers $L^\infty$-norms.
\end{remark}
\begin{proof}[Proof of Lemma~\ref{lem:epsnec}]
    Let $(\bx_k)_{k\in\N}$ not be $\eps$-rich. Then, by definition, there exists $C_1>0$ such that for all $C_2\gg C_{\rm stab}$ there exists $\eps>0$ and an architecture $\chi$ with $G\in\WW(\chi)$ such that $\#\chi\leq C_1\#\chi(\eps)$, $\LL(G)\leq \eps$, and $\GG(G)> C_2\eps$.
    We construct a second training problem with the same data $\bx_k$ but different response function $\widetilde y(x):= G(x)$. The training algorithm
    applied to this modified problem will only see a slight $\eps$ change in responses $\widetilde y$ and thus, by assumption~\eqref{eq:stabletraining}, produce a similar training result $\widetilde F\in\W(\widetilde \chi)$.
    This, however, implies that the generalization error $\widetilde \GG$ with respect to the modified problem satisfies
    \begin{align*}
        \sqrt{\widetilde \GG(\widetilde F)}+ \sqrt{\GG( F)}&\geq 
         \sqrt{\int_{\XX} |\widetilde F(x) - G(x) - (F-y(x))|^2\,d\mu}
         \\
         &\geq 
         \sqrt{\int_{\XX} |G(x)-y(x)|^2\,d\mu}
         -\sqrt{\int_{\XX} |\widetilde F(x) -F|^2\,d\mu}\\
         &\geq 
 \sqrt{\int_{\XX} |G(x)-y(x)|^2\,d\mu}
         -\sqrt{C_{\rm stab}}|\bG -\by| \geq \sqrt{C_2\eps}-\sqrt{C_{\rm stab}\eps},
    \end{align*}
    since $\LL(G)=|\bG -\by|^2$.
    Thus, $2\max\{\widetilde \GG(\widetilde F),\GG(F)\}\gtrsim C\eps\geq C\max\{\widetilde \LL(\widetilde F),\LL(F)\} $. This concludes the proof.
\end{proof}

\begin{definition}\label{def:optgen}
    We say that a training algorithm $\A$ has optimal generalization if for any sequence of $\eps$-rich training data $(\bx_k)_{k\in\N}$, there exists a constant $C>0$ such that for all $\eps>0$, the output $(\chi,F)$ of $\A(\bx_n,\by,\eps)$ satisfies
    \begin{align*}
        \GG(F)\leq C\eps.
    \end{align*}
\end{definition}

\begin{theorem}
   
    Assume that $\#\chi(\eps)\geq C \eps^{-\widetilde\beta}$ for all $\eps>0$ and constants $C,\widetilde\beta>0$. If Theorem~\ref{thm:main} holds with the same $\widetilde\beta$, then Algorithm~\ref{alg:adaptive} has optimal generalization.
\end{theorem}
\begin{proof}
    Let $\chi_\ell$ denote the architecture and $F_\ell\in\RR(\chi_\ell)$ the corresponding network generated by Algorithm~\ref{alg:adaptive} in the $\ell$-th step. According to Theorem~\ref{thm:main}, given $\eps>0$, there exists $\ell\in\N$ with 
    \begin{align*}
        \LL(F_\ell)\leq \eps
    \quad\text{and}\quad\#\chi_\ell = \#\WW_\ell\leq C_\beta\eps^{-\widetilde\beta}\leq  C_\beta C^{-1} \#\chi(\eps),
    \end{align*}
    where $C_\beta>0$ denotes the constant from Theorem~\ref{thm:main}.
    Assuming an $\eps$-rich sequence of training data, this implies $\GG(F_\ell)\leq C_2\eps$ and hence concludes the proof.
\end{proof}
Note that $\#\chi(\eps)\simeq \eps^{-\beta}$ only implies the existence of networks $F$ that satisfy $\GG(F)\lesssim (\#F)^{-\beta}$. This, however, does not necessarily imply that Theorem~\eqref{thm:main} holds with $\beta$ even if $\LL(F)\lesssim \GG(F)$. The reason is that Theorem~\ref{thm:main} requires networks that can be added to the current iteration $\RR(\WW_\ell)$ and improve the loss, which is a stronger assumption. Whether one can remove this gap is still an open question.

\section{Numerical examples}
In all the numerical examples below, we are interested in minimizing the loss on the training data. For comparison with classical (polynomial) approximation methods, we plot the error
\begin{align*}
    {\rm error}:=\sqrt{\LL(\WW_\ell)}
\end{align*}
instead of the loss itself. \rev{The examples are implemented in \emph{PyTorch} and ran on a NVIDIA A40 GPU with 48 GB of memory. }
 
 \rev{For all computations, we compare Algorithm~\ref{alg:adaptive} with a direct training approach. Algorithm~\ref{alg:adaptive} uses 7000 training epochs in Step~1 and trains $300N$ candidates of $\WW^\star$ with 1000 epochs each to approximately solve~\eqref{eq:Wstar}, where $N$ is the current number of weights in the network. The candidate search is processed in parallel batches on the GPU. For a fair comparison with conventional training (subsequently called \emph{direct training}), we randomly initialize networks (Gaussian random numbers with variance adjusted to the number of input nodes) of given sizes and train from scratch for the same amount of time as Algorithm~\ref{alg:adaptive}.
 This results in up to $\approx 1.5\cdot 10^6$ training epochs for direct training, far more than Algorithm~
 \ref{alg:adaptive} receives in Step~1. Including the training for all candidates in the inner loop of Step~(2) of Algorithm~\ref{alg:adaptive}, Algorithm~\ref{alg:adaptive} receives more total epochs, however, the direct training is inherently sequential and cannot be parallelized across optimizer steps and hence a comparison with equal time budget is fair. All optimization is done with \emph{PyTorch}'s Adam optimizer with learning rates varying between $10^{-3}$ and $10^{-4}$ for direct training and Step~(1) of Algorithm~\ref{alg:adaptive}, and $10^{-2}$ for Step~(2). If the error decays consistently, the learning rate is multiplied by $1.07$ and multiplied by $0.1$ if the error spikes. }

 To reduce the noise coming from the nature of stochastic gradient descent and random initialization, each data point shown in the experiments below is the average of 10 independent training runs.

 \subsection{Shallow networks}
We start with the simple task of approximating $f(x,y):= (x+y)^2$ on the unit square.
We compare Algorithm~\ref{alg:adaptive} with a standard approach of training networks of a given size from scratch. All networks have one hidden layer and only vary in width, i.e., $\chi=(2,w,1)$ for some $w\in \N$. The starting point for Algorithm~\ref{alg:adaptive} is $w=2$ and for the optimization in~\eqref{eq:Wstar}, we use
 $\chi^\star=(2,3,1)$.
 
 The idea behind this architecture is that it can represent $g(\bz\cdot x)$, where $\bz\in\R^2$ is an arbitrary vector and $g \colon \R\to \R$ is a standard hat function. If $\widetilde f\colon \R\to \R$ is a piecewise linear interpolation of $x^2$ 
 with $k$ nodes, the architecture $\chi=(2,3k,1)$ can thus represent the approximation $\widetilde f(x+y)\approx f(x,y)$ exactly. It is well known that $\norm{\widetilde f-f}{L^2([0,1]^2)}=\mathcal{O}(k^{-2})$ for interpolation. Hence, we expect $\sqrt{\LL(\WW_\ell)}=\mathcal{O}(\#\chi^{-2})$ with optimal training. 
 
 Figure~\ref{fig:conv1} confirms this rate and shows that the hierarchical strategy of Algorithm~\ref{alg:adaptive} is clearly superior to the direct approach.

 \medskip

We repeat the experiment on $[0,1]^{10}$ with the function $f(\bx):=(\sum_{i=1}^{10} x_i)^2/10$. The starting architecture is $\chi=(10,2,2,1)$ and we apply Step~(2) of Algorithm~\ref{alg:adaptive} only to the final layer of the network keeping the first layer fixed for the inner loop. The idea behind this architecture is that the network should first learn the direction of importance, in this case $\bz=(1,\ldots,1)$, and then approximate $x\mapsto x^2$. Compared  to the one layer approach, this saves a significant number of parameters. Figure~\ref{fig:conv1} confirms the behavior of Algorithm~\ref{alg:adaptive} even for this high-dimensional example.
Again, we compare the hierarchical approach with direct training of networks with architecture $\chi^\star=(2,10,2,w,1)$ for $w\in\{10,20,40\}$ and confirm the superiority of Algorithm~\ref{alg:adaptive}.

\medskip

Note that the network architectures of the direct approach and the hierarchical approach of Algorithm~\ref{alg:adaptive} are identical and both training methods could find the same weights. The experiments in this section show, however, that with random initialization the hierarchical training will reliably produce optimal approximation results, while the direct approach seems to get stuck in local minima.

\begin{figure}
\begin{center}
     \includegraphics[width=0.49\textwidth]{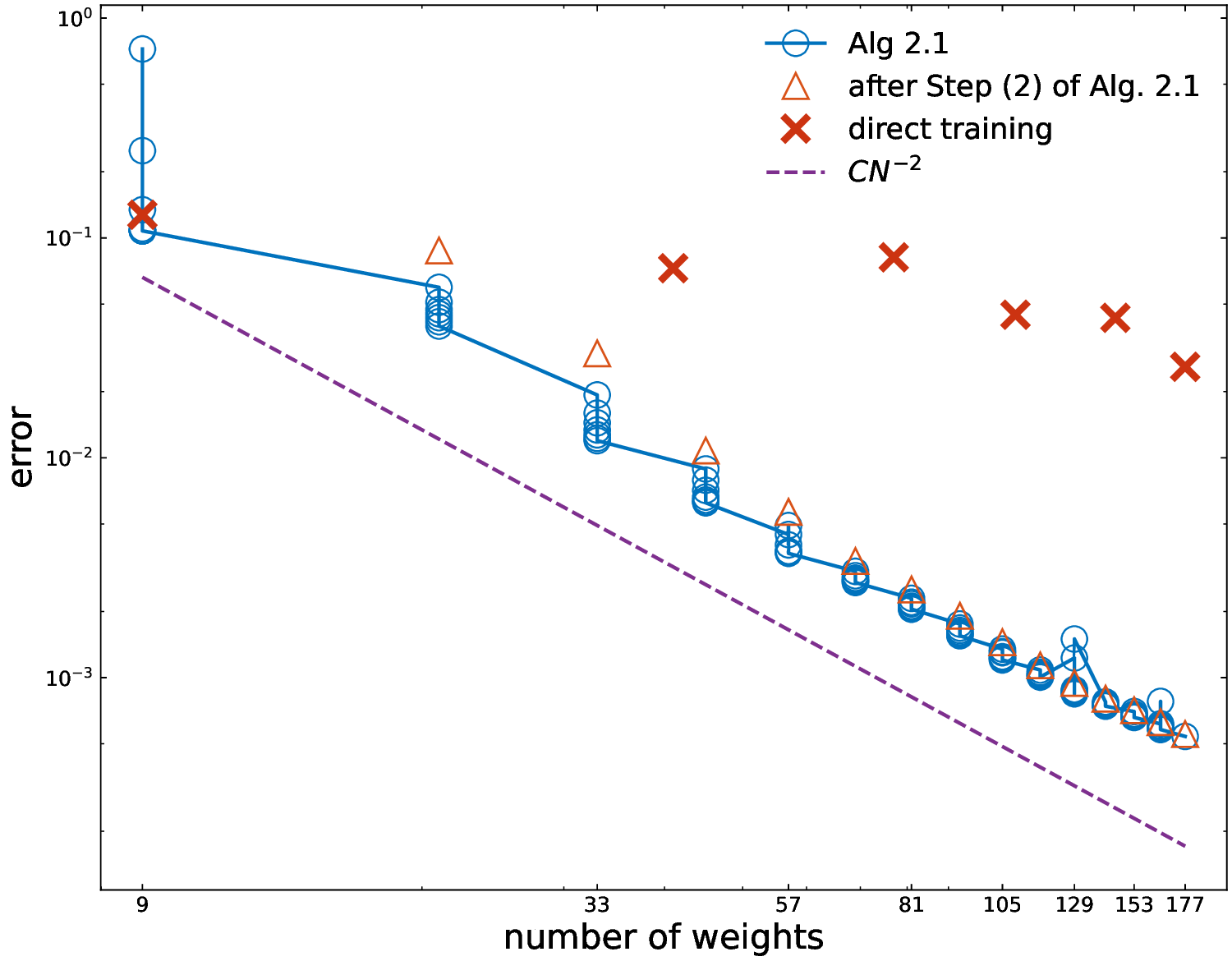}
     \includegraphics[width=0.49\textwidth]{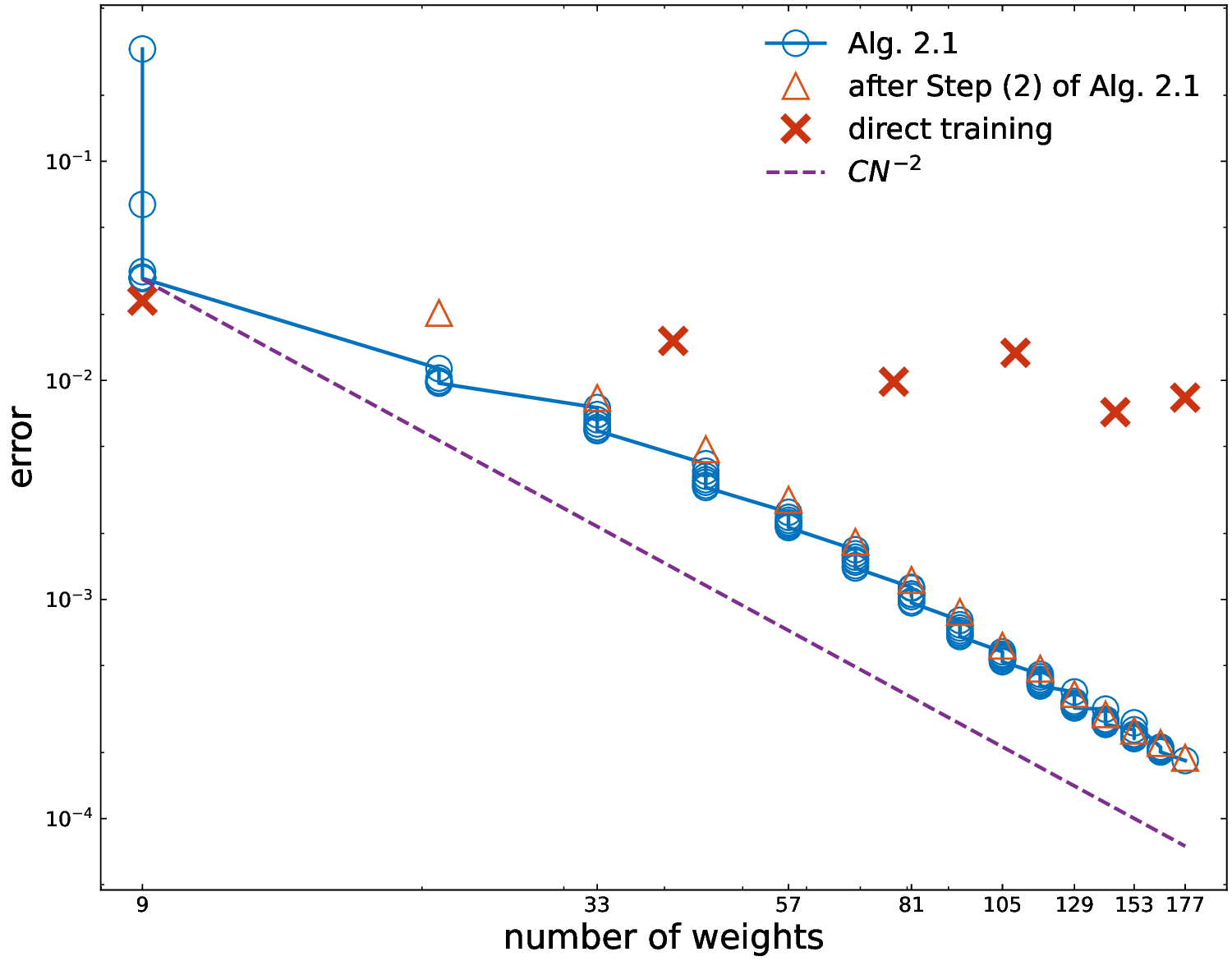}\\
     \includegraphics[width=0.49\textwidth]{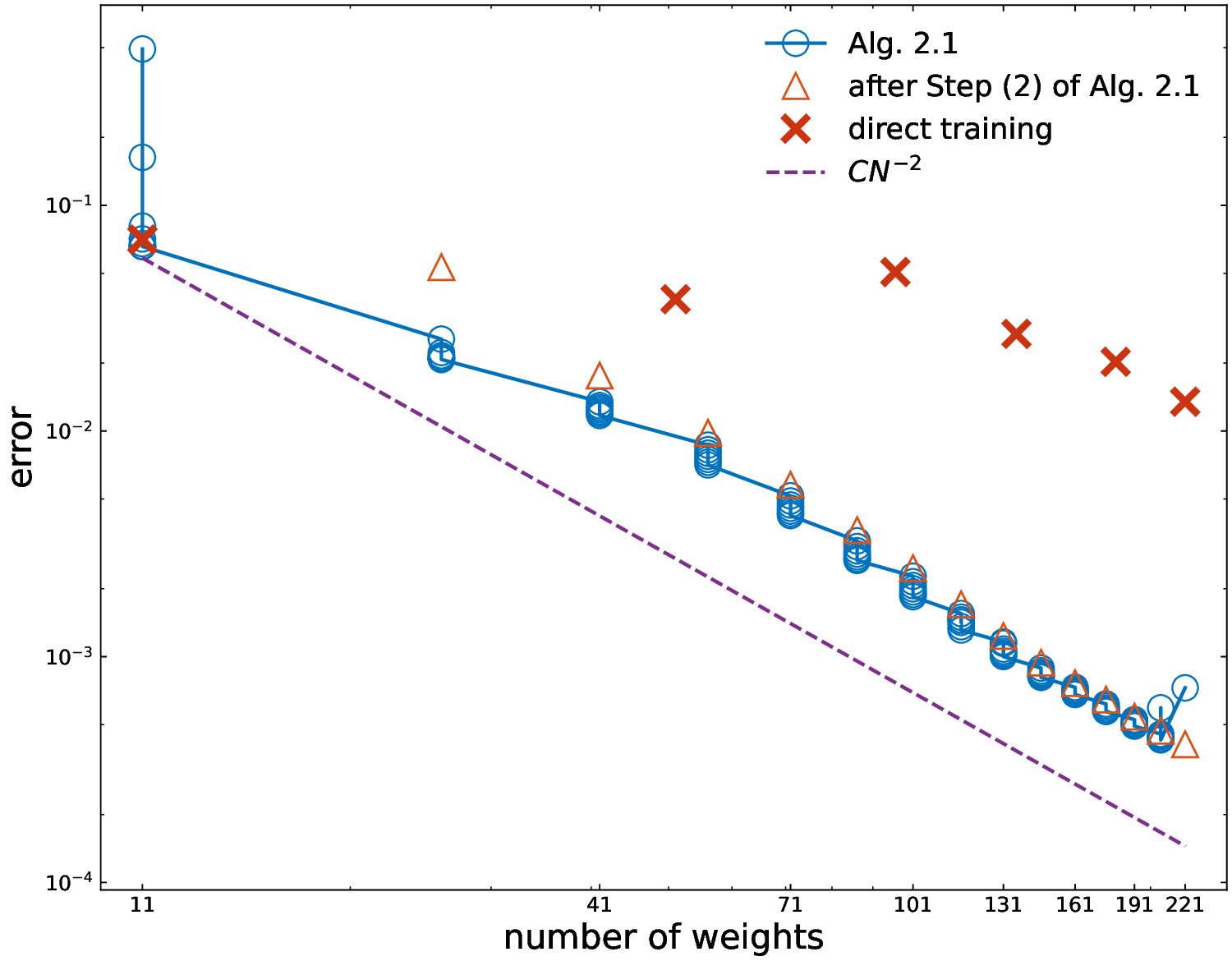}
     \includegraphics[width=0.49\textwidth]{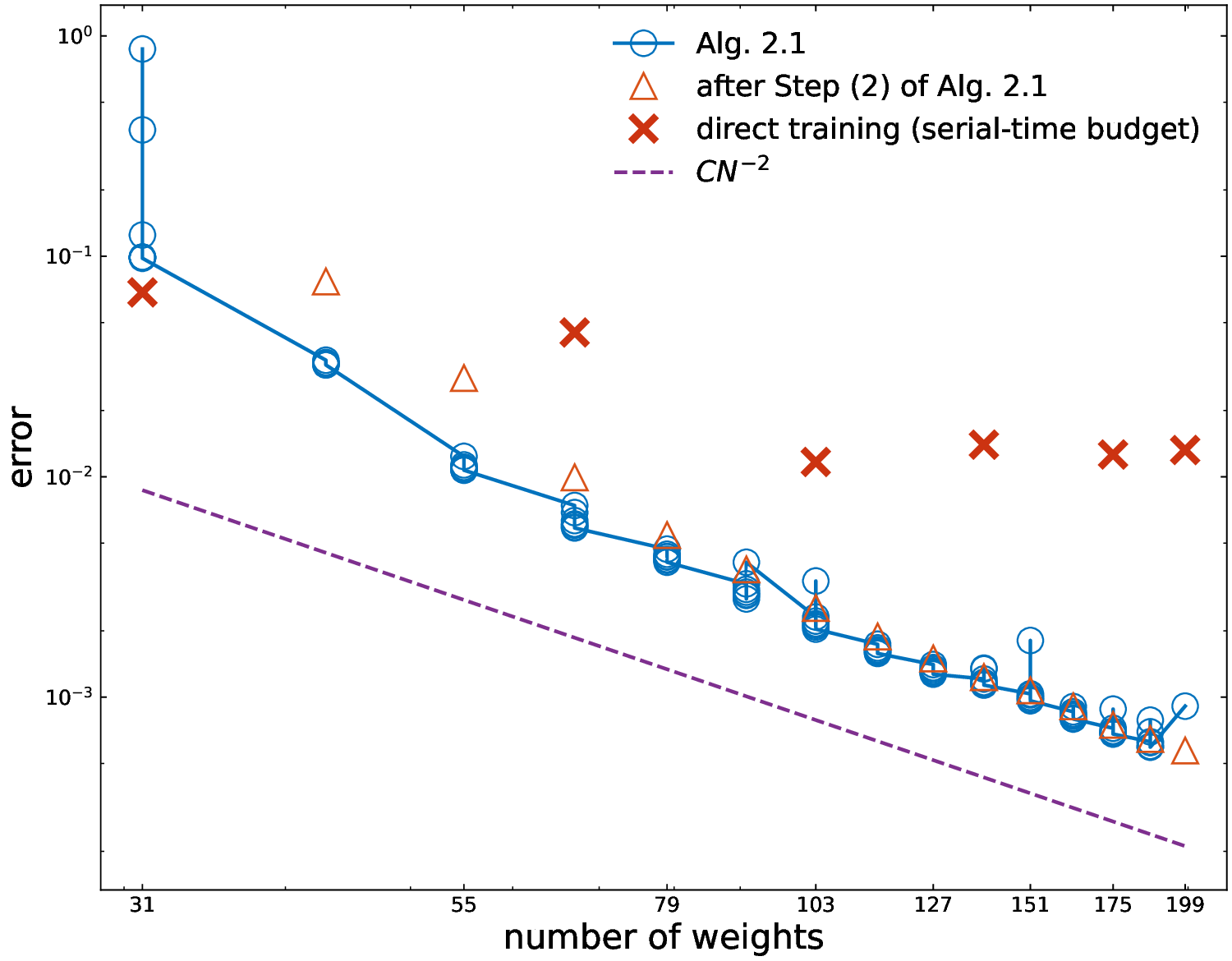}
    \caption{Comparison of hierarchical and direct training for $f(x,y)=(x+y)^2$, $f(x,y)=(x+y)^{2/3}$, $f(x,y,z )=(x+y+z)^2/3$, and $f(x)=(\sum_{i=1}^{10}x_i)^2/10$ (top left to bottom right). We plot the average error over ten training runs. The dashed line represents $\mathcal{O}(n^{-2})$.}
    \label{fig:conv1}
    \end{center}
\end{figure}

\subsection{Deep networks}
\rev{To demonstrate that the algorithm can effectively train deeper networks, we try to \emph{rediscover} the classical construction by Yarotsky~\cite{xsquared} to approximate $x-x^2$ on $[0,1]$. Arguably, this result forms the basis of almost the entire literature of deep learning approximation results. With the hat function $g(x) = 2x -4\max\{0,x-1/2\}$, the construction is
\begin{align*}
    x-x^2  = \sum_{j=1}^\infty 4^{-j}g^{\circ j}(x)\quad\text{for }x\in[0,1],
\end{align*}
where $g^{\circ j}=g\circ g\circ \cdots \circ g$ denotes the $j-1$-times composition of  $g$ with itself, i.e., $g^{\circ 1}=g$. Truncating the series at $j\leq N$, it is well-known that the approximation error in $L^\infty([0,1])$ is bounded by $4^{-N}$. This construction is inherently unstable to perturbations of $g$. However, when $g$ is fixed, perturbation of the weights $4^{-j}$ is stable in the sense of Definition~\ref{def:Lstable}.}

\rev{We modify the inner loop in Step~(2) of Algorithm~\ref{alg:adaptive} slightly to not only add $\RR(\WW^\star)$ to the existing network but also allow it to compose the existing network with a new layer. Thus, Step~(2b) in Algorithm~\ref{alg:adaptive} updates the network structure such that
\begin{align*}
\RR(\WW_{\ell+1}) = \alpha\RR(\WW_\ell)\circ \RR(\WW_1^\star) + \RR(\WW_2^\star)
\end{align*}
instead of only allowing $\RR(\WW_{\ell+1}) = \RR(\WW_\ell)+ \RR(\WW^\star)$. This modification requires us to replace~\eqref{eq:Wstar} with~\eqref{eq:WstarDirect}. The rest of the algorithm remains unchanged. This is outside the setting of Theorem~\ref{thm:main}, but still in the general spirit of Algorithm~\ref{alg:adaptive}. For $\chi^\star$, we choose an architecture that allows us to represent two different hat functions in $\WW_i^\star$, $i=1,2$, which can be done with $\chi^\star$ that contains six parameters (location of peak, width, and scale for each hat function).
It could be done with one hat function according to the theoretical construction above, but we aim to demonstrate that exact knowledge of the optimal architecture is not important. We observe in Figure~\ref{fig:conv2}, that the hierarchical training indeed achieves exponential convergence with respect to the trainable parameters. This is even stronger than what Theorem~\ref{thm:main}, which is inherently limited to algebraic convergence rates, could predict. Analyzing the construction above, we see that the convergence rate (using the minimal three parameters required to represent $g$) w.r.t. number of weights $M$ is ${\rm error} = \mathcal{O}(4^{-N})=\mathcal{O}(4^{-M/3}) = \mathcal{O}(e^{-0.462 M})$. In Figure~\ref{fig:conv2}, we see that Algorithm~\ref{alg:adaptive} achieves a similar exponential rate (note that Theorem~\ref{thm:main} makes no statement about the exponential rates). Interestingly, it seems that the direct training does not converge at all.}
\begin{figure}
\begin{center}
    \includegraphics[width=\textwidth]{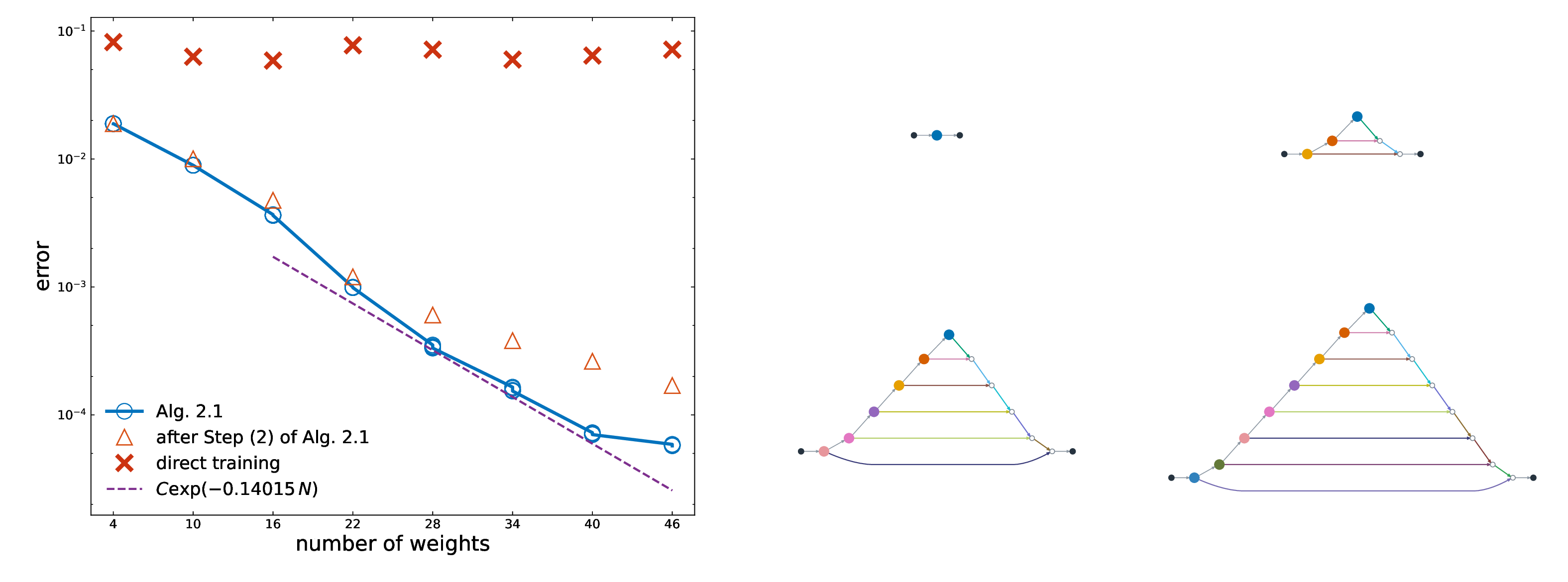}
    \caption{ Comparison of hierarchical and direct training for $f(x) = x-x^2$. We plot the average error over ten training runs. On the right-hand side, we plot the network structures of the networks with 4, 16, 36, 46 weights, where a colored dot represents $\chi^\star$ and an empty circle represents summation.}
    \label{fig:conv2}
    \end{center} 
\end{figure}

\rev{We extend this experiment to the function $x^3-x^2$. Inspired by the representation $2x^3 = (x^2+x)^2 -(x^2)^2 - x^2$, we initialize $\WW_0$ of Algorithm~\ref{alg:adaptive} with $F_1\circ F_2 + F_1\circ F_1 + F_1$, where $F_i$ represent networks that each start with architecture $\chi_0=(1,3,1)$ and are extended as for the previous experiment using the modified Algorithm~\ref{alg:adaptive}. We perform the optimization~\eqref{eq:Wstar} separately for $F_1$ and $F_2$ and only extend the network that achieves the largest error reduction. Again, we observe in Figure~\ref{fig:cubic} that Algorithm~\ref{alg:adaptive} achieves exponential convergence rates, while direct training struggles to converge at all.}
\begin{figure}
\begin{center}
     \includegraphics[width=\textwidth]{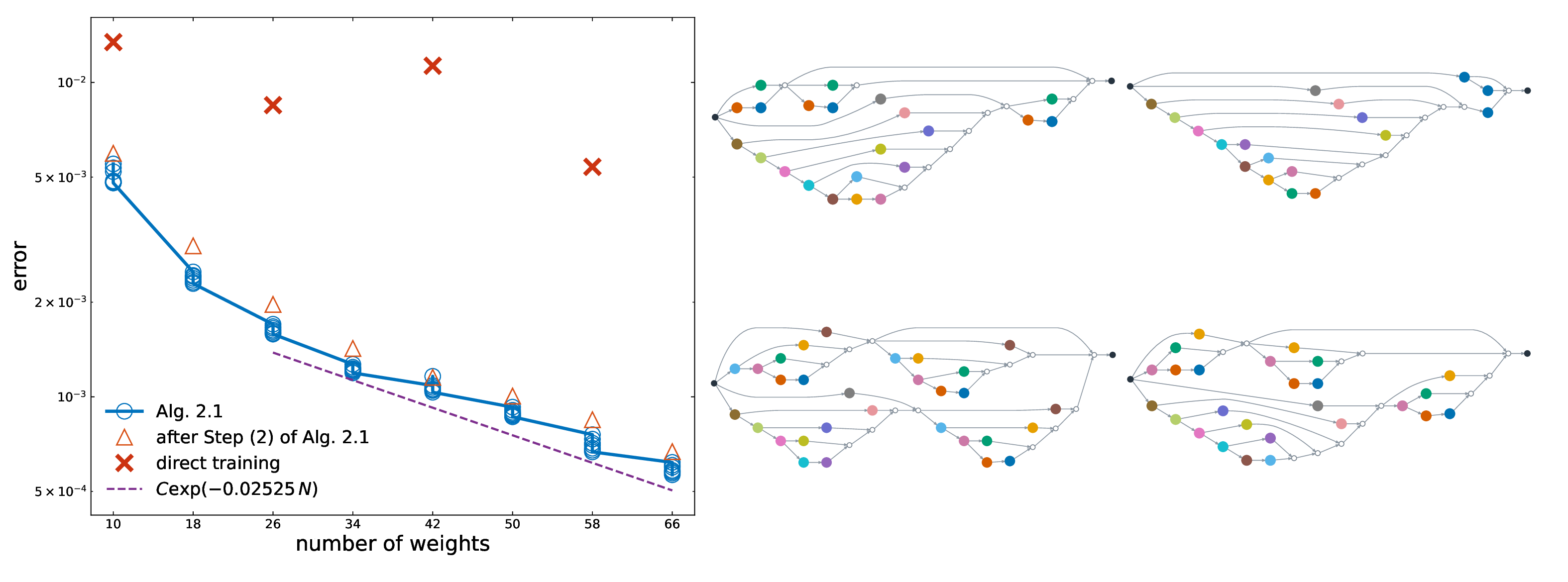}
    \caption{Comparison of hierarchical and direct training for $f(x)=x^3-x^2$. We plot the average error over ten training runs. On the right-hand side, we plot the four different networks structures that emerged during the 10 runs of Algorithm~\ref{alg:adaptive}, where a colored dot represents $\chi^\star$ and an empty circle represents summation.}
    \label{fig:cubic}
    \end{center}
\end{figure}

\bibliographystyle{plain}
\bibliography{literature}
\end{document}

%% file: mymacros.tex
\def\A{\mathbb A}

\def\R{{\mathbb R}}
\def\N{{\mathbb N}}
\def\E{{\mathbb E}}

\def\GG{{\mathcal G}}

\def\RR{{\mathcal R}}

\def\LL{{\mathcal L}}

\def\XX{{\mathcal X}}

\def\WW{{\mathcal W}}

\def\W{{\mathfrak{W}}}

\def\bF{\boldsymbol{F}}
\def\bG{\boldsymbol{G}}

\def\bRR{\boldsymbol{\RR}}
\def\bx{\boldsymbol{x}}
\def\by{\boldsymbol{y}}
\def\bz{\boldsymbol{z}}

\def\b1{\boldsymbol{1}}

\def\bB{{\boldsymbol{B}}}
\def\bW{{\boldsymbol{W}}}

\def\norm#1#2{\|#1\|_{#2}}

\def\set#1#2{\big\{#1\,:\,#2\big\}}

\def\eps{\varepsilon}

\def\normL2#1#2{\|#1\|_{L^2(#2)}}

\newtheorem{assumption}[theorem]{Assumption}
\newtheorem{remark}[theorem]{Remark}